\documentclass[11pt]{amsart}

\RequirePackage{etex}
\usepackage[utf8]{inputenc}
\usepackage[T1]{fontenc}
\usepackage[english]{babel}

\usepackage[mathscr]{euscript}
\usepackage{caption}
\usepackage{subcaption}
\usepackage{datetime}
\usepackage{url}
\usepackage{verbatim}
\usepackage{float}
\usepackage{lipsum}

\usepackage{titletoc}
\usepackage{textcomp}
\usepackage{anysize}
\usepackage{fancyhdr}
\usepackage{calrsfs}

\usepackage{amssymb}
\usepackage{amsmath,mathtools}
\usepackage{amsfonts}
\usepackage{amscd}
\usepackage{amsthm}

\usepackage{latexsym}
\usepackage{enumerate}
\usepackage[shortlabels]{enumitem}
\usepackage{array}
\usepackage{stackrel}
\usepackage{stmaryrd}
\usepackage{stackengine}

\usepackage{graphicx}
\graphicspath{ {Images/} }
\usepackage{color,graphics}
\usepackage{tikz}
	\usetikzlibrary{arrows,automata,patterns, decorations.pathmorphing, decorations.pathreplacing, decorations.markings} 
\usepackage{tikz-cd}
	\usetikzlibrary{matrix,arrows,calc,automata,trees}
\usepackage[all,cmtip]{xy}
\usepackage{tkz-graph}
\usepackage{pgfplots,caption}
\pgfplotsset{compat=1.9}

\usepackage{lscape}
\usepackage[active]{srcltx}
\usepackage{rotating}

\usepackage[hypertexnames=false]{hyperref}

\makeatletter
\renewcommand{\tocsection}[3]{%
  \indentlabel{\@ifnotempty{#2}{\bfseries\ignorespaces#1 #2\quad}}\bfseries#3}
\renewcommand{\tocsubsection}[3]{%
  \indentlabel{\@ifnotempty{#2}{\ignorespaces#1 #2\quad}}#3}

\newcommand\@dotsep{4.5}
\def\@tocline#1#2#3#4#5#6#7{\relax
  \ifnum #1>\c@tocdepth 
  \else
    \par \addpenalty\@secpenalty\addvspace{#2}%
    \begingroup \hyphenpenalty\@M
    \@ifempty{#4}{%
      \@tempdima\csname r@tocindent\number#1\endcsname\relax
    }{%
      \@tempdima#4\relax
    }%
    \parindent\z@ \leftskip#3\relax \advance\leftskip\@tempdima\relax
    \rightskip\@pnumwidth plus1em \parfillskip-\@pnumwidth
    #5\leavevmode\hskip-\@tempdima{#6}\nobreak
    \leaders\hbox{$\m@th\mkern \@dotsep mu\hbox{.}\mkern \@dotsep mu$}\hfill
    \nobreak
    \hbox to\@pnumwidth{\@tocpagenum{\ifnum#1=1\bfseries\fi#7}}\par
    \nobreak
    \endgroup
  \fi}
\AtBeginDocument{%
\expandafter\renewcommand\csname r@tocindent0\endcsname{0pt}
}
\def\l@subsection{\@tocline{2}{0pt}{2.5pc}{5pc}{}}
\makeatother

\fancypagestyle{my_own_style}{
		\fancyhf{}
		\fancyfoot[C]{Page \thepage}
}

\newcommand{\N}{{\mathbb N}}
\newcommand{\Z}{{\mathbb Z}}

\newcommand{\C}{{\mathbb C}}
\newcommand{\R}{{\mathbb R}}

\newcommand{\Ps}{{\mathbb P}}
\newcommand{\V}{{\mathbb V}}
\newcommand{\K}{{\mathbb K}}

\DeclareMathAlphabet{\pazocal}{OMS}{zplm}{m}{n}

\newcommand{\calA}{{\pazocal A}}
\newcommand{\calB}{{\pazocal B}}

\newcommand{\calF}{{\pazocal F}}
\newcommand{\calG}{{\pazocal G}}

\newcommand{\calM}{{\pazocal M}}
\newcommand{\calN}{{\pazocal N}}
\newcommand{\calO}{{\pazocal O}}
\newcommand{\calP}{{\pazocal P}}

\newcommand{\calR}{{\pazocal R}}
\newcommand{\calS}{{\pazocal S}}

\newcommand{\calU}{{\pazocal U}}

\newcommand{\pazR}{{\mathcal R}}

\newcommand{\gotA}{{\mathfrak A}}

\newcommand{\gotR}{{\mathfrak R}}

\newcommand{\ra}{\rightarrow}

\newcommand{\xra}{\xrightarrow}

\newcommand{\act}{\curvearrowright}

\newcommand{\ol}{\overline}
\newcommand{\ul}{\underline}
\newcommand{\wh}{\widehat}

\newcommand{\hugezero}{\mbox{\normalfont\Huge\bfseries 0}}
\newcommand{\bigzero}{\mbox{\normalfont\Large\bfseries 0}}

\newcommand{\rk}{\operatorname{rk}}
\newcommand{\Rk}{\operatorname{Rk}}

\newcommand{\soc}{\operatorname{soc}}

\newcommand{\xddots}{%
  \raise 4pt \hbox {.}
  \mkern 6mu
  \raise 1pt \hbox {.}
  \mkern 6mu
  \raise -2pt \hbox {.}
}

\numberwithin{equation}{section}

\theoremstyle{plain}

\newtheorem{theorem}{Theorem}[section]
\newtheorem*{theorem*}{Theorem}
\newtheorem{lemma}[theorem]{Lemma}
\newtheorem{proposition}[theorem]{Proposition}

\theoremstyle{definition}

\newtheorem{definition}[theorem]{Definition}

\newtheorem{remark}[theorem]{Remark}
\newtheorem*{remark*}{Remark}

\newtheorem*{assumption*}{Assumption}
\newtheorem{observation}[theorem]{Observation}

\newtheorem{hypothesis}[theorem]{Hypothesis}
\newtheorem{notation}[theorem]{Notation}

\title[Sylvester rank functions crossed products]{Sylvester matrix rank functions on crossed products}
\author{Pere Ara}
\address[P. Ara]{Departament de Matem\`atiques, Universitat Aut\`onoma de Barcelona, 08193 Bellaterra (Barcelona), Spain.}
\email{para@mat.uab.cat}
\author{Joan Claramunt}
\address[J. Claramunt]{Departament de Matem\`atiques, Universitat Aut\`onoma de Barcelona, 08193 Bellaterra (Barcelona), Spain.}
\email{jclaramunt@mat.uab.cat}


\thanks{Both authors were partially supported by DGI-MINECO-FEDER through the grant MTM2017-83487-P and by the Generalitat de Catalunya through the grant 2017-SGR-1725. 
The second named author was also partially supported by DGI-MINECO-FEDER through the grant BES-2015-071439.}

\date{\today}

\begin{document}

\pagestyle{plain}
 
\begin{abstract}
In this paper we consider the algebraic crossed product $\calA := C_K(X) \rtimes_T \Z$ induced by a homeomorphism $T$ on the Cantor set $X$, where $K$ is an arbitrary field and $C_K(X)$ denotes 
the $K$-algebra of locally constant $K$-valued functions on $X$. 
We investigate the possible Sylvester matrix rank functions that one can construct on $\calA$ by means of full ergodic $T$-invariant probability measures $\mu$ on $X$. 
To do so, we present a general construction of  an approximating sequence of $*$-subalgebras $\calA_n$ which are embeddable into a (possibly infinite) product of matrix algebras over $K$. 
This enables us to obtain a specific embedding of the whole $*$-algebra $\calA$ into $\calM_K$, the well-known von Neumann continuous factor over $K$, 
thus obtaining a Sylvester matrix rank function on $\calA$ by restricting the unique one defined on $\calM_K$. This process gives a way to obtain a Sylvester matrix rank function 
on $\calA$, unique with respect to a certain compatibility property concerning the measure $\mu$, namely that the rank of a characteristic function of a clopen subset $U \subseteq X$ must 
equal the measure of $U$.
\end{abstract}

\maketitle

\tableofcontents

\normalsize

\section{Introduction and motivation}\label{section-introduction.motivation}

Sylvester matrix rank functions have been widely studied in different contexts. For a C*-algebra $\gotA$, any tracial state $\tau$ on $\gotA$ gives rise to a Sylvester matrix rank function 
defined by the rule $\text{rk}_{\tau} (A) = \lim_{n\to \infty} \tau (|A|^{1/n})$ for every matrix $A$ over $\gotA$ (\cite{BH}). For von Neumann regular rings, these functions were already 
studied by von Neumann, and they were afterwards studied in depth by I. Halperin, K. R. Goodearl and D. Handelman, amongst others, see \cite[Chapters 16--21]{Goo91} for a systematic 
exposition of this theory. For general rings, they were introduced first by Malcolmson \cite{Mal}, and were used to characterize ring homomorphisms to division rings and simple artinian rings \cite{Mal, Sch}. 
There has been a recent flurry of studies on Sylvester rank functions, in connection, amongst other subjects, with the strong Atiyah Conjecture and the L\"uck approximation Conjecture, 
see \cite{Elek16, Elek17, HanfengLi, Jaik-survey, Jaik, JaLo, JL}.

One of our main motivations on writing down this paper comes from the following well-known theorem in the theory of $C^*$-algebras. 
Recall that the celebrated Murray-von Neumann Theorem (\cite{MvN43}) states that all the hyperfinite $II_1$ factors on separable, 
infinite-dimensional Hilbert spaces are $*$-isomorphic. Let now $G$ be a countable discrete, amenable group acting on a compact metrizable 
space $X$, and let $\mu$ be an ergodic, full and $G$-invariant probability measure on $X$, with respect to which the action is essentially free.  
Denote by $\varphi_{\mu}$ the corresponding extremal tracial state on $C(X) \rtimes_r G$ defined by
$$\varphi_{\mu}(a) = \int_X E(a) d\mu , \quad a \in C(X) \rtimes_r G,$$
where $E : C(X) \rtimes_r G \ra C(X)$ is the canonical conditional expectation onto $C(X)$. Then one can embed $C(X) \rtimes_r G$ 
inside the hyperfinite $II_1$ factor $\pazR$ in such a way that $\varphi_{\mu}$ extends to the unique tracial state $\tau_{\pazR}$.

We seek to obtain analogous results in an algebraic setting, by replacing traces by Sylvester matrix rank functions, and weak completions by rank completions. 
To attain this goal we will develop an internal construction, based on the work of Putnam et al \cite{P89,P90,HPS92}. More concretely, given a homeomorphism $T$ on a totally disconnected, 
compact metrizable space $X$ and an arbitrary field $K$, we consider the $K$-algebra $C_K(X)$ of locally constant $K$-valued functions on $X$, and the algebraic crossed product $\calA := C_K(X)\rtimes_T \mathbb Z$. 
(Note that $C_K(X)$ is the algebra of continuous functions $X\to K$, where $K$ has the {\it discrete topology}.) We then show that there exists a large subalgebra $\calA_{\infty}$ of $\calA$ which embeds into a von Neumann regular algebra $\mathfrak R_{\infty}$. The large subalgebra $\calA_{\infty}$ is a union of a nested sequence of subalgebras $\calA_n$, each of which can be embedded in an algebra $\gotR_n$ which is a (possibly infinite) direct product of matrix algebras over $K$. We show that there are compatible embeddings $\gotR_n \hookrightarrow \gotR_{n+1}$, so that the algebra $\calA_{\infty} = \bigcup_{n=1}^{\infty} \calA_n$ embeds in the direct limit algebra $\gotR_{\infty} = \lim \gotR_n$. In particular, the algebra $\mathfrak R_{\infty}$ is von Neumann regular, and we show that it admits a rank function $\rk_{\gotR_{\infty}}$ such that $\mu (U) = \rk_{\gotR_{\infty}}(\pi_{\infty} (\chi_U))$ for all clopen subsets $U$ of $X$, where $\pi_{\infty}\colon \calA_{\infty} \to \gotR_{\infty}$ is the canonical embedding. The algebra $\calA$ itself does not embed in $\gotR_{\infty}$, but it does embed in the rank completion $\gotR_{{\rm rk}}$ of $\gotR_{\infty}$ with respect to $\rk_{\gotR_{\infty}}$. Using this, we show that there exists a unique Sylvester matrix rank function $\rk_{\calA}$ on $\calA$ such that $\rk_{\calA} (\chi_U)= \mu (U)$ for all clopen subsets $U$ of $X$. We moreover show that $\rk_{\calA} \in \partial_e \mathbb P (\calA)$, the set of extreme 
points of the compact convex set $\mathbb P (\calA)$ of all the Sylvester matrix rank functions on $\calA$, and using a recent result by the authors \cite{AC2}, we identify the 
algebra $\gotR_{{\rm rk}}$ with the continuous von Neumann factor $\calM_K$, in complete analogy with the analytic setting described above. 

In the final section, using the close relation between Sylvester matrix rank functions on $\calA$ and $T$-invariant measures on $X$, we show that every Sylvester matrix rank function on $\calA$ is regular in the sense of \cite{Jaik}, that is, it is induced from a ring homomorphism from $\calA$ to a regular ring.

Another main motivation for this work is the possibility to obtain results related to the study of $L^2$-invariants for group algebras. Indeed, the results in this paper can be applied to obtain approximations of a large class of group algebras and their $*$-regular closures, as follows.  

Assume that $H$ is a countable discrete, torsion abelian group and that $\alpha$ is an automorphism of $H$. One may then consider the semidirect product $G:=H\rtimes _{\alpha}\Z$. 
The dual group $\widehat{H}$ is a compact, totally disconnected metrizable topological space and, by using Pontryagin duality, one has $K[H] \cong C_K(\widehat{H})$ for any field $K$ 
whose characteristic is coprime with all the orders of the elements of $H$ and containing all the $n$-th roots of unity for those $n$ which 
appear as orders of elements of $H$ (see \cite[Section 3.1]{Cthesis} for a detailed exposition). The action of $\Z$ on $H$ induces a canonical action $\widehat{\alpha}$ of $\Z$ by 
homeomorphisms on $\widehat{H}$, and one obtains that the above isomorphism extends to an isomorphism of $K$-algebras
$$K[G]=K[H\rtimes_{\alpha} \Z] \cong K[H] \rtimes_{\alpha} \Z \cong C_K(\widehat{H})\rtimes_{\widehat{\alpha}} \Z.$$
Hence the group algebras $K[G]$ can be analyzed using the tools developed in this paper. For instance, the group algebra of the lamplighter group $\Z_2\wr \Z$ arises from this construction by taking the shift automorphism on $H=\bigoplus_{\Z} \Z_2$ (see subsection \ref{subsection-lamplighter.group} for details).
 
Recall that if $G$ is a discrete group, the $*$-regular algebra $\calU(G)$ is the algebra of unbounded operators affiliated to the group von Neumann algebra $\calN(G)$. The algebra $\calU(G)$ is a $*$-regular ring, endowed with a canonical rank function $\text{rk}$, defined by $\rk (a) = \text{tr} (s(a))$, where $s(a) \in \calN(G)$ is the support projection of $a$ and $\text{tr}$ is the canonical trace on $\calN(G)$. The $*$-regular algebra $\calU(G)$ and the $*$-regular closure $\calR_{K[G]}:=\calR(K[G], \calU(G))$ of $K[G]$ in $\calU(G)$ play a fundamental role in the study of the Atiyah Problem, see \cite{Luck}, \cite{Jaik-survey} for details. Indeed, a result of Jaikin-Zapirain \cite[Corollary 6.2]{Jaik} implies the equality 
$$\phi(K_0(\calR_{K[G]})) = \calG(G, K),$$
where $\phi$ is the state on $K_0(\calR_{K[G]})$ induced by the canonical rank function on $\calU(G)$, and $\calG(G,K)$ is the subgroup of $\R$ generated by the $l^2$-Betti numbers arising from matrices over $K[G]$. 

If $G = H \rtimes _{\alpha} \Z$ is as above, and $K$ is a subfield of the complex numbers $\C$ closed under complex conjugation and containing enough roots of unity, one can use the approach developed in this paper to obtain suitable approximations of the $*$-regular closure $\calR_{K[G]}$.  This will be developed in the forthcoming paper \cite{AC}.

This paper is structured as follows. In Section \ref{section-background.preliminaries} we collect the definitions of von Neumann regular rings and pseudo-rank functions, 
together with the notion of Sylvester matrix rank functions. We also recall the concept of a $*$-regular ring. 
In Section \ref{section-main.construction}, we give the main construction of the article, so given a full ergodic $T$-invariant measure $\mu$ on $X$ and given a partition of $X$ into clopen subsets, we obtain an approximating
subalgebra which can be embedded in a possibly infinite direct product of matrix algebras over $K$ (see Proposition \ref{proposition-embedding.B.R}). We also briefly study, in subsection \ref{subsection-lamplighter.group}, 
a motivating example, the lamplighter group algebra $K[\Z_2\wr \Z]$.   
We use the above construction in Section \ref{section-Sylvester.rank.functions.A}, together with the main result in \cite{AC2}, to obtain 
an embedding of $\calA = C_K(X) \rtimes_{\alpha} \Z$ into the well-known von Neumann 
continuous factor $\calM _K$ (Theorem \ref{theorem-completion.A.Ay}, Proposition \ref{proposition-unique.rank}, Theorem \ref{theorem-vN.cont.factor}). 
As a consequence, we get a faithful extremal Sylvester matrix rank function on $\calA$, and prove a uniqueness 
statement for such a rank function provided that a suitable compatibility condition with $\mu$ is satisfied.  Section \ref{section-space.PA} is devoted to the study of the structure of the compact convex set $\Ps(\calA)$ 
consisting of all Sylvester matrix rank functions on $\calA$. In particular, we show that all such rank functions are regular (Theorem \ref{theorem-regular.Sylvester.space}).

\section{Background and preliminaries}\label{section-background.preliminaries}

Here we collect background definitions, concepts, and results needed during the course of the paper.

\subsection{Von Neumann regular rings and pseudo-rank functions}\label{subsection-regular.rank.functions}

A unital ring $R$ is called a \textit{regular ring} if for every element $x \in R$ there exists $y \in R$ such that $x = xyx$. Note that, in this case, the element $e = xy$ is an idempotent 
and generates the same (right) ideal as $x$. In fact, a characterization for regular rings is that every finitely generated one-sided ideal of $R$ is generated by a single 
idempotent (see \cite[Theorem 1.1]{Goo91}). Regularity is closed under taking extensions, ideals\footnote{Since the definition of regularity on a unital ring does not concern the unit itself, 
the notion of a regular ideal is analogous: for any element $x$ of the ideal, there exists another element $y$, also in the ideal, such that $x = xyx$.}, direct products, matrices, direct limits, among others.

Two idempotents $e,f \in R$ are said to be \textit{equivalent}, denoted by $e \sim f$, if there exists an isomorphism $eR \cong fR$ as right $R$-modules. Equivalently, $e \sim f$ if there exist elements 
$x \in eRf$, $y \in fRe$ such that $e = xy$ and $f = yx$. 

We now introduce the notion of pseudo-rank functions on a regular ring $R$.
\begin{definition}\label{Ch1-definition-rank.function}
A \textit{pseudo-rank function} on a (regular) ring is a real-valued function $\rk : R \ra [0,1]$ satisfying the following properties:
\begin{enumerate}[a)]
\item $\rk(0) = 0$, $\rk(1) = 1$.
\item $\rk(xy) \leq \rk(x),$ $\rk(y)$ for every $x,y \in R$.
\item If $e, f$ are orthogonal idempotents, then $\rk(e+f) = \rk(e) + \rk(f)$.
\end{enumerate}
If $\rk$ satisfies the additional property
\begin{enumerate}[d)]
\item $\rk(x) = 0$ if and only if $x = 0$,
\end{enumerate}
then $\rk$ is called a \textit{rank function} on $R$.
\end{definition}
For general properties of pseudo-rank functions over regular rings one can consult \cite[Chapter 16]{Goo91}.

Every pseudo-rank function $\rk$ on a regular ring $R$ defines a pseudo-metric $d$ on $R$ by the rule $d(x,y) = \rk(x-y)$ for $x,y\in R$. If moreover $\rk$ is a rank function, then $d$ is a metric. 
Note that we can always achieve the situation where $d$ is indeed a metric by factoring through the ideal $\text{ker}(\rk)$ of elements having zero rank. Since the ring operations 
are continuous with respect to this metric, one can consider the completion $\ol{R}$ of $R$ with respect to $d$. $\ol{R}$ is again a regular ring, and $\rk$ can be uniquely extended continuously 
to a rank function $\ol{\rk}$ on $\ol{R}$ such that, with the new metric induced by $\ol{\rk}$, $\ol{R}$ is also complete, and coincides with the natural metric on $\ol{R}$ inherited from the 
completion process. It turns out that the completion $\ol{R}$ is also a right and left self-injective ring (see Theorems 19.6 and 19.7 of \cite{Goo91}).

The space of pseudo-rank functions $\Ps(R)$ on a regular ring $R$ is a Choquet simplex (\cite[Theorem 17.5]{Goo91}), and the completion $\ol{R}$ of $R$ with respect to $\rk \in \Ps(R)$ is a simple ring if 
and only if $\rk$ is an extreme point in $\Ps(R)$ (\cite[Theorem 19.14]{Goo91}).

\subsection{Sylvester matrix rank functions}\label{subsection-Sylvester.rank.functions}

Regular rings are also of great interest since every (pseudo-) rank function $\rk$ on $R$ can be uniquely extended to a (pseudo-)rank function on matrices over $R$ (see e.g. \cite[Corollary 16.10]{Goo91}). This is no longer true if we do not assume $R$ to be regular. The definition that seems to fit in the general setting is the notion of Sylvester matrix rank functions.
\begin{definition}\label{Ch1-definition-sylvester.rank}
Let $R$ be a unital ring. A \textit{Sylvester matrix rank function} ${\rm rk}$ on $R$ is a function that assigns a nonnegative real number to each matrix over $R$ and satisfies the following conditions:
\begin{enumerate}[a)]
 \item ${\rm rk} (M)= 0$ if $M$ is a zero matrix, and ${\rm rk}(1)= 1$. 
 \item ${\rm rk} (M_1M_2) \leq {\rm rk}(M_1), {\rm rk}(M_2)$ for any matrices $M_1$ and $M_2$ which can be multiplied.
 \item ${\rm rk} \begin{pmatrix} M_1 & 0 \\ 0 & M_2 \end{pmatrix} = {\rm rk} (M_1) + {\rm rk}(M_2) $ for matrices $M_1$ and $M_2$.
 \item ${\rm rk} \begin{pmatrix} M_1 & M_3 \\ 0 & M_2  \end{pmatrix} \ge {\rm rk}(M_1) + {\rm rk}(M_2)$ for any matrices $M_1$, $M_2$ and $M_3$ of appropriate sizes. 
\end{enumerate}
\end{definition}

For more theory and properties about Sylvester matrix rank functions we refer the reader to \cite{Jaik} and \cite[Part I, Chapter 7]{Sch}.

We denote by $\mathbb P(R)$ the compact convex set of Sylvester matrix rank functions on $R$. It is well-known (see for example \cite[Proposition 16.20]{Goo91}) that, in case $R$ is a regular ring, this space coincides with the space of pseudo-rank functions on $R$.

As in the case of pseudo-rank functions on a regular ring, a Sylvester matrix rank function $\rk$ on a unital ring $R$ gives rise to a 
pseudo-metric by the rule $d(x,y) = \rk(x-y)$ for $x, y \in R$. We call it \textit{faithful} if its kernel $\text{ker}(\rk)$ is exactly $\{ 0 \}$. 
In this case, $d$ becomes a metric on $R$. We can always obtain a faithful Sylvester rank function by passing to the quotient $R \ra R / \text{ker}(\rk)$. 
The ring operations are continuous with respect to this metric, so one can consider the completion $\ol{R}$ of $R$ with respect to $d$, and $\rk$ extends uniquely to a Sylvester matrix 
rank function $\ol{\rk}$ on $\ol{R}$.

\subsection{\texorpdfstring{$*$}{}-regular rings}\label{subsection-*regular.rings}

A \textit{$*$-regular ring} is a regular ring endowed with a proper involution, that is, an involution $*$ such that $x^*x= 0$ implies $x=0$.

The involution is called \textit{positive definite} in case the condition
$$\sum_{i=1}^n x_i^* x_i= 0 \quad \implies \quad x_i=0 \text{ for all } 1 \leq i \leq n$$
holds for each positive integer $n$. If $R$ is a $*$-regular ring with positive definite involution, then $M_n(R)$, endowed with the $*$-transpose involution, is also a $*$-regular ring.

For $*$-regular rings, we have a strong property concerning idempotents generating principal right/left ideals of $R$. In fact, if we demand these idempotents to be 
projections (i.e. elements $e \in R$ such that $e = e^2 = e^*$), then it turns out that there exist unique projections generating a given principal right/left ideal. 
More precisely, for each element $x \in R$ there are unique projections $e, f \in R$ (denoted by $\mathrm{LP}(x)$ and $\mathrm{RP}(x)$ and called the \textit{left} and \textit{right projections} of $x$, respectively) 
such that $xR = eR$ and $Rx = Rf$; moreover, there exists a unique element $y \in f R e$ such that $xy = e$ and $yx = f$, termed the \textit{relative inverse} of $x$.

We refer the reader to \cite{Ara87, Berb, Jaik} for further information about $*$-regular rings.

\section{The main construction}\label{section-main.construction}

We start with some preliminaries. Here, we will concentrate on the most basic dynamical system, the one provided by a single homeomorphism $T : X \ra X$ on a totally disconnected, 
compact metrizable space $X$. Recall that a probability measure $\mu$ on $X$ is \textit{ergodic} if for every $T$-invariant Borel subset $E$ of $X$ we have that 
either $\mu(E) = 0$ or $\mu(E) = 1$, and $\mu$ is said to be \textit{invariant} in case $\mu(T(E)) = \mu(E)$ for every Borel subset $E$ of $X$. 
For instance, it is well-known (cf. \cite[Example 3.1]{KM}) that the product measure $\mu$ on $\{0,1\}^{\Z}$, where we take the $\big(\frac{1}{2},\frac{1}{2}\big)$-measure on $\{0,1\}$, is invariant and ergodic.
 
We will also often assume that our ergodic invariant measure $\mu$ is \textit{full}, that is, $\mu(U) >0$ for every non-empty open subset $U$ of $X$.\footnote{This is not true for a general ergodic 
invariant measure. For instance take $T= \text{Id}$; then every measure is invariant, in particular the 
one-point mass measures are invariant and ergodic.}

The following is a simple application of Rokhlin's Lemma. We include a proof for the convenience of the reader. 

\begin{lemma}\label{lemma-Rokhlin} 
Let $\mu$ be an ergodic $T$-invariant probability measure on $X$, and take $E$ to be a Borel subset of $X$ with positive measure. Consider the first return map $r_E : E \ra \N \cup \{ \infty \}$, defined by 
$$r_E(x) = \mathrm{min} \{ l > 0 \mid T^l(x) \in E \}$$
in case there is $l>0$ such that $T^l(x) \in E$, and $r_E(x) = \infty$ otherwise. For each $k \in \N$, consider $Y_k^l = T^l(r_E^{-1} (k))$, for $0 \leq l \leq k-1$. Also let $Y_{\infty} = E \backslash \bigsqcup_{k \in \N} Y_k^0$ be the set of points of $E$ that do not return to $E$.

Then we have that $T(Y_k^l) = Y_{k}^{l+1}$ for $0 \leq l < k-1$, all the sets $Y_k^l$ are mutually disjoint, and the set 
$$Y= Y(E) = \bigsqcup _{k \geq 1} \bigsqcup_{l=0}^{k-1} Y_k^l$$
satisfies that $\mu(Y)= 1$. In particular, we get $\sum _{k \geq 1} \sum_{l=0}^{k-1} \mu(Y_k^l) = \sum_{k \geq 1} k \mu(Y_k^0) = 1$.
\end{lemma}

\begin{proof}
Note that each $Y_k^0$ is given by the set $E \cap T^{-1}(X \backslash E) \cap \cdots \cap T^{-k+1}(X \backslash E) \cap T^{-k}(E)$, which are Borel sets. Therefore $Y_{\infty} = E \backslash \bigsqcup_{k \in \N} Y_k^0$ is also Borel. Moreover, if we choose $E$ to be a clopen set, then all the $Y_k^l$ are also clopen sets, and $Y_{\infty}$ is a closed set.

By its definition all the sets $Y_k^l, Y_{k'}^{l'}$ and the $T$-translates of $Y_{\infty}$ are pairwise disjoint. Therefore
$$1 \geq \mu \Big( \bigsqcup_{j \geq 0} T^j(Y_{\infty}) \Big) = \sum_{j \geq 0} \mu(T^j(Y_{\infty})) = \sum_{j \geq 0} \mu(Y_{\infty}),$$
which shows that $\mu(Y_{\infty}) = 0$.

Now we set
$$Z := \Big( \bigsqcup_{j \geq 0} T^j(Y_{\infty}) \Big) \sqcup \Big( \bigsqcup_{k \geq 1} \bigsqcup_{l=0}^{k-1} Y_k^l \Big).$$
We observe that $T(Z) \subseteq Z$. Indeed, it is clear that $T(Y_k^l) \subseteq Z$ for $0 \leq l < k-1$, and also, $T(Y_k^{k-1}) \subseteq E \subseteq Z$. Take $Z_0 = \bigcap_{j \geq 0} T^j(Z) \subseteq Z$. Clearly $T(Z_0) = Z_0$ since $T(Z) \subseteq Z$, so $Z_0$ is a $T$-invariant Borel set. Hence by ergodicity of the measure either $\mu(Z_0) = 0$ or $\mu(Z_0) = 1$. But by invariance and the fact that $T^j(Z) \subseteq Z$ for all $j \geq 0$,
$$\mu(Z \backslash Z_0) = \mu\Big( \bigcup_{j\geq 0} Z \backslash T^j(Z) \Big) \leq \sum_{j \geq 0} \mu(Z \backslash T^j(Z)) = 0,$$
so $\mu(Z) = \mu(Z_0)$ is either $0$ or $1$. Since $\mu(E) > 0$, we get $\mu(Z) = 1$. The result follows by invariance of the measure and the fact that $Y_{\infty}$ is a null-set.
\end{proof}

For any field with involution $K$, we will consider the $*$-algebra $C_K(X)$ of continuous functions from $X$ to $K$, where $K$ is endowed with the discrete topology (i.e., $C_K(X)$ is the algebra of locally constant $K$-valued functions on $X$). A homeomorphism $T$ of $X$ induces an action $\alpha$ of the integers $\Z$ on $C_K(X)$ by 
$$\alpha_n(b)(x) = b(T^{-n}(x))$$
for $b\in C_K(X)$ and $x\in X$. Note that if $U$ is a clopen subset of $X$ and $\chi_U$ denotes the characteristic function of $U$, then $\chi_U \in C_K(X)$ and $\alpha_n(\chi_U) = \chi_{T^n(U)}$ for each $n \in \Z$.

If $A$ is a $*$-algebra and $\alpha$ is a $*$-automorphism of $A$, we define the algebraic crossed product $A \rtimes_{\alpha} \Z$ as the $*$-algebra of formal
finite sums $\sum_{i \in \Z} a_it^i$, where $a_i \in A$. The sum is componentwise, the product is determined by the rule $ta := \alpha(a)t$, and the involution is given by $(at^i)^* := t^{-i}a^* = \alpha^{-i}(a^*) t^{-i}$. If $T$ is a homeomorphism of a totally disconnected metrizable compact space $X$, and $\alpha (=\alpha _1) $ is the $*$-automorphism of $C_K(X)$ defined above, we will denote the crossed product $C_K(X) \rtimes_{\alpha} \Z$ by $C_K(X) \rtimes_T \Z$.

Recall also the definition of a \textit{partial} algebraic crossed product \cite{Exel}. A partial action of $\Z$ on a $*$-algebra $A$ is a pair $\phi = ( \{ A_n \}_{n \in \Z}, \{\phi_n\}_{n \in \Z})$ consisting of a collection of self-adjoint two-sided ideals $A_n$ of $A$ and a collection of $*$-isomorphisms $\phi_n : A_{-n} \to A_n$ such that
\begin{enumerate}[a)]
\item $A_0 = A$ and $\phi_0$ is the identity map, and
\item $\phi_n \circ \phi_m \subseteq \phi_{n+m}$, meaning that $\phi_n \circ \phi_m$ is defined on the largest possible domain where the composition makes sense and $\phi_{n+m}$ extends it.
\end{enumerate}
The partial algebraic crossed product of $A$ by $\Z$ with respect to the partial action $\phi$, denoted by $A \rtimes_{\phi} \Z$, is defined to be the set of all finite formal sums $\sum_{n \in \Z}a_n \delta_n$ 
with $a_n \in A_n$ and $\delta_n$ are indeterminates, with the usual addition and scalar multiplication, and the product defined by the rule $(a_n \delta_n) \cdot (b_m \delta_m) := \phi_n( \phi_{-n}(a_n)b_m) \delta_{n+m}$. The involution is then defined through the rule $(a_n \delta_n)^* := \phi_{-n}(a_n^*) \delta_{-n}$.

We start our construction by first approximating our space $X$, and then using this approximation to construct a family of approximating algebras for $C_K(X) \rtimes_T \Z$. First, some definitions.

\begin{definition}\label{definition-quasi.partitions}
Let $Y$ be a topological space, endowed with a probability measure $\mu$.
\begin{enumerate}[a)]
\item By a \textit{partition} of $Y$ we will understand a finite family $\calP$ of nonempty, pairwise disjoint clopen subsets of $Y$, such that $Y = \bigsqcup_{Z \in \calP}Z$.

\noindent Given two partitions $\calP_1, \calP_2$ of $Y$, we say that $\calP_2$ is finer than $\calP_1$ (or $\calP_1$ is coarser than $\calP_2$), written $\calP_1 \precsim \calP_2$, if every element $Z \in \calP_2$ is contained in a (unique) element $Z' \in \calP_1$, that is $Z \subseteq Z'$.
\item By a \textit{quasi-partition} of $Y$ we will understand a finite or countable family $\ol{\calP}$ of nonempty, pairwise disjoint clopen subsets of $Y$, such that $Y = \bigsqcup_{\ol{Z} \in \ol{\calP}} \ol{Z}$ up to a set of measure $0$, i.e. $\mu \Big( Y \backslash \bigsqcup_{\ol{Z} \in \ol{\calP}} \ol{Z} \Big) = 0$.

\noindent Given two quasi-partitions $\ol{\calP}_1, \ol{\calP}_2$ of $Y$, we say that $\ol{\calP}_2$ is finer than $\ol{\calP}_1$ (or $\ol{\calP}_1$ is coarser than $\ol{\calP}_2$), written $\ol{\calP}_1 \precsim \ol{\calP}_2$, if every element $\ol{Z} \in \ol{\calP}_2$ is contained in a (unique) element $\ol{Z}' \in \ol{\calP}_1$, that is $\ol{Z} \subseteq \ol{Z}'$.
\end{enumerate}
\end{definition}

\noindent Note that, in the hypotheses of Lemma \ref{lemma-Rokhlin}, the family $\{ Y_k^l \mid Y_k^l \neq \emptyset \}$ forms a quasi-partition of $X$.

\begin{definition}\label{definition-approx.alg}
Consider the $*$-algebra $\calA := C_K(X) \rtimes_T \Z$. Let $E$ be a nonempty clopen subset of $X$, and let $\calP$ be a partition of $X \backslash E$. Define $\calB := \calA (E, \calP)$ as the unital $*$-subalgebra of $\calA$ generated by the elements $\chi_Z \cdot t$ for $Z \in \calP$.
\end{definition}

Our first goal is to express $\calB$ as a partial algebraic crossed product by a $\Z$-action. Let $\calB_0 = C_K(X) \cap \calB$, which is a commutative $*$-subalgebra of $\calB$. We first give a complete description of $\calB_0$ in terms of characteristic functions.

\begin{lemma}\label{lemma-comm.B0}
The $*$-algebra $\calB_0$ is linearly spanned by $1$ and the projections of the form
\begin{equation}\label{equation-projections.B0}
\chi_{T^{-r}(Z_{-r})\cap T^{-r+1}(Z_{-r+1}) \cap \cdots \cap Z_0\cap T(Z_1)\cap \cdots \cap T^{s-1}(Z_{s-1})},
\end{equation}
where $Z_{-r}, ..., Z_0, ..., Z_{s-1} \in \calP$, and $r, s \geq 0$. 
 \end{lemma}

\begin{proof}
Recall that for a clopen subset $U$ of $X$, $t \chi_U t^{-1} = T(\chi_U) = \chi_{T(U)}$. We have 
$$(\chi_{Z_0}t)(\chi_{Z_1}t)\cdots (\chi_{Z_{s-1}}t)(\chi_{Z_{s-1}}t)^{*} \cdots (\chi_{Z_0}t)^{*} = \chi_{Z_0 \cap T(Z_1) \cap \cdots \cap T^{s-1}(Z_{s-1})}$$
and
$$(\chi_{Z_{-1}}t)^* \cdots (\chi_{Z_{-r}}t)^*(\chi_{Z_{-r}}t) \cdots (\chi_{Z_{-1}}t) = \chi_{T^{-r}(Z_{-r}) \cap \cdots \cap T^{-1}(Z_{-1})},$$
which shows that all the projections of the form \eqref{equation-projections.B0} belong to $\calB_0$.

Let $\calF$ be the set of projections of the form \eqref{equation-projections.B0} together with $0$. Observe that the family $\calF$ is closed under products. 
Hence, to show the result, it is enough to prove that any product of generators $a_1 \cdots a_n$ of degree $0$ in $t$ belongs to $\calF$ (here each $a_i$ is either of the form $\chi_Z t$ or of the form $(\chi_Z t)^* = t^{-1} \chi_Z$, for $Z \in \calP$). An immediate observation is that if $a_1 \cdots a_n$ is of degree $0$, then $n$ must be even. We will proceed to show the result by induction on $n$.

Clearly the result is true for $n = 2$, so assume $n > 2$ is even and that each product of at most $n-2$ generators of degree $0$ in $t$ belongs to $\calF$. Define $d(i) \in \Z$ by $d(i)= \text{deg}_t(a_1 \cdots a_i)$. Suppose, for instance, that $d(1) =1$.

If there is $r < n$ such that $d(r) =0$, then since
$$0 = \text{deg}_t(a_1 \cdots a_n) = d(r) + \text{deg}_t(a_{r+1} \cdots a_n) = \text{deg}_t(a_{r+1} \cdots a_n),$$
we can use induction to conclude that the products $a_1 \cdots a_r$ and $a_{r+1} \cdots a_n$ belong to $\calF$, hence the whole product $a_1 \cdots a_n$ belongs to $\calF$ too.

Otherwise we must have $d(r) > 0$ for all $r < n$, and since $d(n)=0$ we must have that $\text{deg}_t(a_n) = -1$. Then necessarily $a_1 = \chi_{Z_1}t$ and $a_n= t^{-1}\chi_{Z_2}$ for some $Z_1, Z_2 \in \calP$, and thus
\begin{equation}\label{equation-product.B0}
a_1 a_2 \cdots a_{n-1} a_n = \chi_{Z_1} (ta_2\cdots a_{n-1}t^{-1}) \chi_{Z_2}.
\end{equation}
By induction, the product $a_2 \cdots a_{n-1}$ belongs to $\calF$, and hence $t a_2 \cdots a_{n-1}t^{-1}$ 
either belongs to $\calF$ or it is of the form $\chi_{T(Z'_1) \cap \cdots \cap T^{s-1}(Z'_{s-1})}$, for $Z'_1,\dots ,Z'_s\in \calP$. 
Therefore, the product \eqref{equation-product.B0} is zero if $Z_1 \neq Z_2$ and, if $Z_1 = Z_2$, it belongs to $\calF$. In either case $a_1 \cdots a_n$ belongs to $\calF$, as desired. 
The case where $d(1) = -1$ is similar. 
 
It then follows that $\calB_0$ is the linear span of the given set of projections $\calF$. This concludes the proof of the lemma.
\end{proof}

We now consider the structure of $\calB$ as a partial algebraic crossed product by $\Z$ on $\calB_0$. Note that we can write $\calB = \bigoplus_{i \in \Z} \calB_i t^i$, where $\calB_i = \chi_{X \setminus (E \cup T(E) \cup \cdots \cup T^{i-1}(E))} \calB_0$ and $\calB_{-i} = \chi_{X \setminus (T^{-1}(E) \cup \cdots \cup T^{-i}(E))} \calB_0$ for $i > 0$. 

Observe that if $b_i t^i \in \calB$ then $b_i = \chi_{X \backslash (E \cup \cdots \cup T^{i-1}(E))} b_i$ and $b_{-i} = \chi_{X \backslash (T^{-1}(E) \cup \cdots \cup T^{-i}(E))} b_{-i}$ for $i > 0$, and so
$$b_i t^i = b_i (\chi_{X \backslash E} t)^i, \quad b_{-i} t^{-i} = b_{-i} (t^{-1} \chi_{X \backslash E})^i \quad \text{for positive $i$}.$$

In particular, it is true that $\calB_i t^i = \calB_0 (\chi_{X \backslash E} t)^i$ and $\calB_{-i}t^{-i} = \calB_0 (t^{-1} \chi_{X \backslash E})^i$ for $i > 0$.

\begin{observation}\label{observation-approx.t}
One needs to be careful with the term $\chi_{X \backslash E} t$ because, although $t$ is invertible with inverse $t^{-1} = t^*$, this is not true for $\chi_{X \backslash E} t$. As a consequence, equalities like
$$(\chi_{X \backslash E} t)^i = (\chi_{X \backslash E} t)^{i+j} (\chi_{X \backslash E} t)^{-j}$$
are no longer true and even meaningful for $i > j > 0$. In the next lemma we summarize the basic arithmetics that one can achieve with these powers.
\end{observation}

From now on for $i > 0$, we will write $(\chi_{X \backslash E} t)^{-i}$ for the element $(t^{-1} \chi_{X \backslash E})^i$. We will also understand that $(\chi_{X \backslash E} t)^0$ is $1$.

\begin{lemma}\label{lemma-approx.t}
Fix $i \geq j > 0$. We have the following rules:
\begin{enumerate}[i)]
\item $(\chi_{X \backslash E} t)^i = (\chi_{X \backslash E} t)^{i-j} (\chi_{X \backslash E} t)^j = (\chi_{X \backslash E} t)^j (\chi_{X \backslash E} t)^{i-j}$.
\item $(\chi_{X \backslash E} t)^{-i} = (\chi_{X \backslash E} t)^{-i+j} (\chi_{X \backslash E} t)^{-j} = (\chi_{X \backslash E} t)^{-j} (\chi_{X \backslash E} t)^{-i+j}$.
\item $(\chi_{X \backslash E} t)^i \neq (\chi_{X \backslash E} t)^{i+j} (\chi_{X \backslash E} t)^{-j} \neq (\chi_{X \backslash E} t)^{-j} (\chi_{X \backslash E} t)^{i+j} \neq (\chi_{X \backslash E} t)^i$, but: we have the first equality when multiplied (to the left) by the projection $\chi_{X \backslash (E \cup \cdots \cup T^{i+j-1}(E))}$; we have the second equality when multiplied (to the left) by the projection $\chi_{X \backslash (T^{-j}(E) \cup \cdots \cup T^{i+j-1}(E))}$; we have the third equality when multiplied (to the left) by the projection $\chi_{X \backslash (T^{-j}(E) \cup \cdots \cup T^{-1}(E))}$.
\item $(\chi_{X \backslash E} t)^{-i} \neq (\chi_{X \backslash E} t)^{-i-j} (\chi_{X \backslash E} t)^j \neq (\chi_{X \backslash E} t)^j (\chi_{X \backslash E} t)^{-i-j} \neq (\chi_{X \backslash E} t)^{-i}$, but: we have the first equality when multiplied (to the right) by the projection $\chi_{X \backslash (T^{-j}(E) \cup \cdots \cup T^{-1}(E))}$; we have the second equality when multiplied (to the right) by the projection $\chi_{X \backslash (T^{-j}(E) \cup \cdots \cup T^{i+j-1}(E))}$; we have the third equality when multiplied (to the right) by the projection $\chi_{X \backslash (E \cup \cdots \cup T^{i+j-1}(E))}$.
\end{enumerate}
\end{lemma}

The proof of Lemma \ref{lemma-approx.t} is purely computational, so we will omit it. From now on, we will make use of it without any further reference.

Note that each $\calB_i, \calB_{-i}$ is an ideal of $\calB_0$. Let us define the basic map of the partial action of $\Z$ on $\calB_0$ as conjugation by $\chi_{X \backslash E} t$:
$$\varphi_1 \colon \calB_{-1} \to \calB_1, \quad \varphi_1(b_{-1}) = (\chi_{X \setminus E}t) b_{-1} (\chi_{X \setminus E}t)^*,$$
which is a $*$-isomorphism between $\calB_{-1} = \chi_{X \setminus T^{-1}(E)} \calB_0$ and $\calB_1 = \chi_{X \setminus E} \calB_0$ with inverse given by conjugation by $(\chi_{X \backslash E}t)^{-1} = t^{-1}\chi_{X \backslash E}$. Note that since $b_{-1} \in \calB_{-1}$, $b_{-1} = \chi_{X \backslash T^{-1}(E)} b_{-1}$. In general for $i \neq 0$ we have a $*$-isomorphism $\varphi_i$ from $\calB_{-i}$ onto $\calB_i$ which is given by conjugation by $(\chi_{X\setminus E}t)^i$. Using these maps we build a partial action $\varphi$ of $\Z$ on $\calB_0$, and we get the following result.

\begin{proposition}\label{proposition-cross.prod.B0}
There is a canonical $*$-isomorphism $\calB_0 \rtimes_{\varphi} \Z \cong \calB$ given by
$$\Psi \colon \calB_0 \rtimes_{\varphi} \Z \to \calB, \quad \Psi\big( \sum_{i \in \Z} b_i \delta_i \big) = \sum_{i \in \Z} b_i (\chi_{X \backslash E} t)^i = \sum_{i \in \Z}b_it^i,$$
where $b_i \in \calB_i$ for $i \in \Z$. 
\end{proposition}
\begin{proof}
Routine. The only nontrivial thing may be to check that the products are preserved. The key observation here is that the product $b_i T^i(b_j)$ belongs to $\calB_{i+j}$ for any integer values of $i,j$; this follows by a case-by-case analysis using Lemma \ref{lemma-approx.t}. After that, a direct computation shows that the products are indeed preserved under $\Psi$:
\begin{align*}
\Psi((b_i \delta_i) & (b_j \delta_j)) = \Psi( \varphi_i( \varphi_{-i}(b_i) b_j ) \delta_{i+j} ) = \varphi_i( \varphi_{-i}(b_i) b_j ) (\chi_{X \backslash E} t)^{i+j} \\
& = (\chi_{X \backslash E}t)^i (\chi_{X \backslash E}t)^{-i} b_i (\chi_{X \backslash E}t)^i b_j (\chi_{X \backslash E}t)^{-i} (\chi_{X \backslash E} t)^{i+j} \\
& = b_iT^i(b_j) (\chi_{X \backslash E} t)^{i+j} = b_iT^i(b_j) t^{i+j} = \Psi(b_i \delta_i) \Psi(b_j \delta_j).\hfill\qedhere
\end{align*}
\end{proof}


We summarize in the next lemma the structure of the elements belonging to the ideals $\calB_i$, $i \in \Z$, of $\calB_0$.

\begin{lemma}\label{lemma-structure.bi}
A nonzero element $b_i \in \calB_i$ ($i \in \Z$) can be written as an orthogonal linear combination  of characteristic functions of nonempty sets of the following four different types:
\begin{align*}
& (I) \quad T^{-N}(Z_{-N}) \cap \cdots \cap T^{M-1}(Z_{M-1}), \text{ with } Z_j \in \calP; \\
& (II) \quad T^{-N}(Z_{-N}) \cap \cdots \cap T^{s-2}(Z_{s-2}) \cap T^{s-1}(E), \text{ for } 0 \leq s \leq M, Z_j \in \calP; \\
& (III) \quad T^{-r}(E) \cap T^{-r+1}(Z_{-r+1}) \cap \cdots \cap T^{M-1}(Z_{M-1}), \text{ for } 0 \leq r \leq N, Z_j \in \calP; \\
& (IV) \quad T^{-r}(E) \cap T^{-r+1}(Z_{-r+1}) \cap \cdots \cap T^{s-2}(Z_{s-2}) \cap T^{s-1}(E), \text{ for } 0 \leq r \leq N, 0 \leq s \leq M, Z_j \in \calP;
\end{align*}
for some $N, M \geq 0$, where if $i < 0$ then $N \geq -i$, and $r \geq -i$ in (III) and (IV), and if $i > 0$ then $M \geq i$, and $s \geq i$ in (II) and (IV).
\end{lemma}
\begin{proof}
Due to Lemma \ref{lemma-comm.B0}, we can write a given $b_i \in \calB_i$ as a sum
$$b_i = \lambda_0 + \sum_S \lambda_S \chi_S$$
where the sets $S$ are of the form \eqref{equation-projections.B0}, and $\lambda_0,\lambda_S \in K$. Note that if $i < 0$ then we can take $\lambda_0 = 0$ and all the sets $S$ as in \eqref{equation-projections.B0} having $r \geq -i$, and similarly if $i > 0$ we can take $\lambda_0 = 0$ and all the sets $S$ as in \eqref{equation-projections.B0} having $s \geq i$.

Take $N$ to be the maximum value of the $r$'s while running through the sets $S$, and $M$ to be the maximum value of the $s$'s. The idea is to expand the element $1$ as an orthogonal sum of characteristic functions using the partition $\calP$. So for a fixed set $S = T^{-r}(Z_{-r}) \cap \cdots \cap T^{s-1}(Z_{s-1})$ with $0 \leq r < N$, we can decompose its characteristic functions as an orthogonal sum as follows:
$$\chi_S = \chi_{T^{-r-1}(E) \cap S} + \sum_{Z \in \calP} \chi_{T^{-r-1}(Z_{-r-1}) \cap S}.$$
By further expanding the set $T^{-r-1}(E) \cap S$ to the right, we will end up with a sum of terms of types (III) and (IV); by expanding $T^{-r-1}(Z_{-r-1}) \cap S$ to both sides we will end up with a sum of terms of all types. Of course, we discard the empty sets that appear in this process. Also, if one of the terms appearing in the expansion of $S$ coincides with another term appearing in the expansion of some other set $S'$, we simply collect them by summing the corresponding coefficients. Proceeding in this way, we will end up with an orthogonal sum of the desired form.
\end{proof}

\subsection{Quasi-partitions and a \texorpdfstring{$*$}{}-representation of \texorpdfstring{$\calB$}{}}\label{subsection-quasi.partitions.representation.B}

Let $X$ be an infinite, totally disconnected metrizable compact space, $T$ a homeomorphism of $X$, and $\mu$ a full ergodic $T$-invariant probability measure on $X$. We apply the previous considerations given in Lemma \ref{lemma-Rokhlin} to the clopen set $E$, and we add into the picture the partition $\calP$ of $X \backslash E$. That is, we consider the coarsest quasi-partition $\ol{\calP}$ of $X$ such that
\begin{enumerate}[a)]
\item $\calP \cup \{ E \} \precsim \ol{\calP}$ and $\{Y_k^l\mid Y_k^l \ne \emptyset \} \precsim \ol{\calP}$, where $\{Y_k^l\mid Y_k^l \ne \emptyset \}$ is the quasi-partition introduced above in Lemma \ref{lemma-Rokhlin}, and
\item if $\ol{Z} \in \ol{\calP}$ and $\ol{Z} \subseteq Y_k^0$ for some $k$, then all its translates belong to the quasi-partition too, that is $T^i(\ol{Z}) \in \ol{\calP}$ for every $1 \leq i \leq k-1$.
\end{enumerate}
$\ol{\calP}$ can be obtained by refining, using $\calP \cup \{E \}$, the quasi-partition $\{Y_k^l\mid Y_k^l \ne \emptyset \}$. It turns out that all the characteristic functions $\chi_{\ol{Z}}$, with $\ol{Z} \in \ol{\calP}$, belong to $\calB$.

\begin{lemma}\label{lemma-quasi.partition}
The quasi-partition $\ol{\calP}$ above consists exactly of all the nonempty subsets of $X$ of the form
\begin{equation}\label{equation-Wexpression}
W = E\cap T^{-1}(Z_1) \cap T^{-2}(Z_2) \cap \cdots \cap T^{-k+1}(Z_{k-1})\cap T^{-k}(E)
\end{equation}
for some $k\ge 1$ and some $Z_1,\dots , Z_{k-1}\in \calP$, together with all their $T$-translates $T^l(W)$, $0 \leq l \leq k-1$. Moreover, each characteristic function $\chi_{\ol{Z}}$ belongs to $\calB$ for any $\ol{Z} \in \ol{\calP}$.
\end{lemma}

\begin{proof}
Let $\V$ denote the set of all the nonempty sets $W$ of the form \eqref{equation-Wexpression}, and let $\calP'$ be the family of all the translates of all $W \in \V$. For $W = E \cap T^{-1}(Z_1) \cap \cdots \cap T^{-k+1}(Z_{k-1}) \cap T^{-k}(E) \in \V$ we define $|W| := k$, the length of $W$. 
We will prove that $\calP' = \ol{\calP}$. We first show:
\begin{enumerate}[(1)]
\item $\calP'$ is a quasi-partition of $X$. Clearly, the sets in $\calP'$ are mutually disjoint since $\calP$ forms a 
partition of $X \backslash E$, and the nonempty sets of $\V$ form, for a fixed length $k$, a partition of $Y_k^0=r_E^{-1}(\{ k \})$. Indeed,
\begin{align*}
\bigsqcup_{\substack{W \in \V \\ |W| = k}} W & = \bigsqcup_{Z_1,...,Z_{k-1} \in \calP} E \cap T^{-1}(Z_1) \cap \cdots \cap T^{-k+1}(Z_{k-1})\cap T^{-k}(E) \\
& =  E \cap (X \backslash T^{-1}(E)) \cap \cdots \cap (X \backslash T^{-k+1}(E)) \cap T^{-k}(E) = Y_k^0.
\end{align*}
As a consequence, for a fixed $0 \leq l \leq k-1$, the $T^l$-translates of the $W \in \V$ having length $k$ form a partition of $Y_k^l = T^l(Y_k^0)$. Since, by Lemma \ref{lemma-Rokhlin}, the family $\{Y_k^l\mid Y_k^l \ne \emptyset \}$ 
forms a quasi-partition of $X$, this shows that $\calP'$ is a quasi-partition of $X$.
\item $\calP'$ refines $\calP \cup \{E \}$ and the family $\{Y_k^l\mid Y_k^l \ne \emptyset \}$. This is a direct consequence of part $(1)$.
\item For $\ol{Z} \in \calP'$ with $\ol{Z} \subseteq Y_k^0$ for some $k$, then $T^i(\ol{Z}) \in \calP'$ for each $1 \leq i \leq k-1$. By construction, all the sets $\ol{Z} \in \calP'$ with $\ol{Z} \subseteq Y_k^0$ for some $k$ are the $W \in \V$ having length $k$. It is then clear that all its translates $T^i(W) \in \calP'$ for $1 \leq i \leq k-1$.
\end{enumerate}
This shows that $\ol{\calP} \precsim \calP'$. To show that $\calP' \precsim \ol{\calP}$, we only have to check that if $Y' \subseteq Y_k^0$ is a nonempty clopen set such that for each $1 \leq i \leq k-1$ the translate $T^i(Y')$ is contained in one of the sets of the partition $\calP$, then $Y' \subseteq W$ for some $W \in \V$. But this is clear, since if $T^i(Y') \subseteq Z_i$ for $i=1,...,k-1$ where $Z_i \in \calP$, and $T^k(Y') \subseteq E$, then $Y' \subseteq E \cap T^{-1}(Z_1) \cap \cdots \cap T^{-k+1}(Z_{k-1}) \cap T^{-k}(E)$. Hence $\calP' = \ol{\calP}$.

We now check that $\chi_W$ belongs to $\calB$, where $W$ is as in \eqref{equation-Wexpression}. First observe that  $\chi_E = 1- (\chi_{X\setminus E} t)(\chi_{X\setminus E} t)^*$ and $\chi_{T^{-1}(E)}= 1- (\chi_{X\setminus E} t)^*(\chi_{X\setminus E} t)$ both belong to $\calB$. Now by Lemma \ref{lemma-comm.B0}, we have that $\chi_{T^{-1}(Z_1) \cap \cdots \cap T^{-k+1}(Z_{k-1})} \in \calB$ for $Z_1,Z_2,\dots , Z_{k-1} \in \calP$. Therefore
$$\chi_W = \chi_E \cdot \chi_{T^{-1}(Z_1) \cap \cdots \cap T^{-k+1}(Z_{k-1})} \cdot (t^{-1} \chi_{X\setminus E})^{k-1} \chi_{T^{-1}(E)} (\chi_{X\setminus E} t)^{k-1} \in \calB.$$
Also, for $1 \leq l \leq k-1$, observe that
$$\chi_{T^l(W)} = \chi_{X \backslash (E \cup T(E) \cup \dots \cup T^{l-1}(E))} \cdot \chi_{T^l(W)} = (\chi_{X \backslash E} t)^l \chi_W (t^{-1} \chi_{X \backslash E})^l,$$
and so $\chi_{T^l(W)} \in \calB$ too.
\end{proof}

\begin{proposition}\label{proposition-minimal.central}
For each $W \in \V$, we have $*$-isomorphisms
$$\chi_W \calB \chi_W \cong K, \quad \calB \chi_W \calB \cong M_{|W|}(K).$$
Moreover, the element $h_W := \sum_{l=0}^{|W|-1} \chi_{T^l(W)}$ is a unit in the two-sided ideal $\calB \chi_W \calB$, a central projection in $\calB$, and
$$h_W \calB \cong M_{|W|}(K).$$
In particular, $\chi_W$ is a minimal projection in $\calB$.
\end{proposition}
\begin{proof}
Fix $W \in \V$. We will prove a more general statement, that is $\chi_{T^l(W)} \calB \chi_{T^l(W)} \cong K$ for all $0 \leq l \leq |W|-1$. Write again $\calB = \bigoplus_{i \in \Z} \calB_i t^i = \bigoplus_{i \in \Z} \calB_0 (\chi_{X \backslash E}t)^i$, so that $\chi_{T^l(W)} \calB \chi_{T^l(W)} = \bigoplus_{i \in \Z} \calB_0 \chi_{T^l(W)} (\chi_{X \backslash E}t)^i \chi_{T^l(W)}$. For $i > 0$, note that
$$\chi_{T^{l+i}(W)} \cdot \chi_{X \backslash (T^{l+(i-1)}(E) \cup \cdots \cup T^i(E))} = \chi_{T^{l+i}(W)},$$
and so
\begin{align*}
\chi_{T^l(W)} (\chi_{X \backslash E}t)^i \chi_{T^l(W)} & = \chi_{T^l(W)} \cdot \chi_{X \backslash (E \cup \cdots \cup T^{i-1}(E))} t^i \chi_{T^l(W)} \\
& = \chi_{T^l(W)} \cdot \chi_{T^{l+i}(W)} \cdot \chi_{X \backslash (E \cup \cdots \cup T^{i-1}(E))} t^i \\
& = \chi_{T^l(W)} \cdot \chi_{T^{l+i}(W)} \cdot \chi_{X \backslash (E \cup \cdots \cup T^{l+(i-1)}(E))} t^i = 0, \\
\chi_{T^l(W)} (t^{-1} \chi_{X \backslash E})^i \chi_{T^l(W)} & = \Big( \chi_{T^l(W)} (\chi_{X \backslash E}t)^i \chi_{T^l(W)} \Big) ^* = 0.
\end{align*}
Therefore $\chi_{T^l(W)} \calB \chi_{T^l(W)} = \chi_{T^l(W)} \calB_0$. Using Lemma \ref{lemma-comm.B0} we get that $\chi_{T^l(W)} \calB_0 = K \chi_{T^l(W)}$, so $\chi_{T^l(W)} \calB \chi_{T^l(W)} = K \chi_{T^l(W)} \cong K$.

Now, by means of previous computations, it is straightforward to see that for general $i,j \in \Z$, we have
\begin{equation}\label{equation-matr.unit}
(\chi_{X \backslash E} t)^i \chi_W (t^{-1} \chi_{X \backslash E})^j =
	\begin{cases}
		(\chi_{X \backslash E} t)^i \chi_W (t^{-1} \chi_{X \backslash E})^j & \text{ for } 0 \leq i,j \leq |W| - 1 \\
		0 & \text{ otherwise}
	\end{cases}.
\end{equation}
We then consider
$$e_{ij}(W) := (\chi_{X \backslash E} t)^i \chi_W (t^{-1} \chi_{X \backslash E})^j, \quad 0 \leq i,j \leq |W|-1.$$
Observe that $e_{ll}(W) = (\chi_{X \backslash E} t)^l \chi_W (t^{-1} \chi_{X \backslash E})^l = \chi_{T^l(W)}$ for $0 \leq l \leq |W|-1$, and that the set $\{ e_{ij}(W) \}$ is a complete system of matrix units for $\calB \chi_W \calB$. 
To prove that $h_W = \sum_{l=0}^{|W|-1} \chi_{T^l(W)} = \sum_{l=0}^{|W|-1} e_{ll}(W)$ is indeed a unit for $\calB \chi_W \calB$, we first use \eqref{equation-matr.unit} to write
\begin{align*}
\calB \chi_W \calB & = \bigoplus_{i,j \in \Z} \calB_0 (\chi_{X \backslash E} t)^i \chi_W \calB_0 (t^{-1} \chi_{X \backslash E})^j = \bigoplus_{i,j \in \Z} \calB_0 (\chi_{X \backslash E} t)^i \chi_W (t^{-1} \chi_{X \backslash E})^j \\
& = \bigoplus_{i ,j \geq 0}^{|W|-1} \calB_0 e_{ij}(W) = \bigoplus_{i ,j \geq 0}^{|W|-1} (\calB_0 e_{ii}(W)) e_{ij}(W) = \bigoplus_{i ,j \geq 0}^{|W|-1} K e_{ij}(W)
\end{align*}
where we have used that $\calB_0 e_{ii}(W) = K e_{ii}(W)$. It is now clear that $h_W$ is a unit for $\calB \chi_W \calB$. 
We thus get the desired $*$-isomorphism by sending
\begin{equation}\label{equation-*isomorphism}
\calB \chi_W \calB \ra M_{|W|}(K), \quad e_{ij}(W) \mapsto e_{ij}
\end{equation}
where $\{e_{ij}\}_{0 \leq i,j \leq |W|-1}$ is a complete system of matrix units for $M_{|W|}(K)$.

For the second statement, since $\calB = \bigoplus_{i \in \Z} \calB_0 (\chi_{X \backslash E} t)^i$, it is enough to show that $h_W$ commutes with all the elements $(\chi_{X \backslash E} t)^i$ for $i \in \Z$. By applying the involution, we may assume without loss of generality that $i \geq 1$. By induction, we may further assume that $i = 1$. But for $0 \leq l \leq |W|-1$,
\begin{align*}
e_{ll}(W) \cdot & \chi_{X \backslash E} t = \chi_{T^l(W)} \cdot \chi_{X \backslash E}t \\
& = \begin{cases}
		 \chi_{X \backslash E}t \cdot e_{l-1,l-1}(W) & \text{ if } 1 \leq l \leq |W|-1 \\
		 0 & \text{ otherwise}
		\end{cases} = \begin{cases}
										e_{l,l-1}(W) & \text{ if } 1 \leq l \leq |W|-1 \\
										0 & \text{ otherwise}
									\end{cases},\\
\chi_{X \backslash E}  t & \cdot e_{ll}(W) = (\chi_{X \backslash E} t)^{l+1} \chi_W (t^{-1} \chi_{X \backslash E})^l = \begin{cases}
										 e_{l+1,l}(W) & \text{ if } 0 \leq l \leq |W|-2 \\
										 0 & \text{ otherwise}
								   \end{cases},
\end{align*}
so by summing up over $l$ and doing the change $l+1 = l'$, it is clear that $h_W \cdot \chi_{X \backslash E} t = \chi_{X \backslash E} t \cdot h_W$. The result follows.

For the last part, simply observe that the family $\{e_{ij}(W)\}_{0 \leq i,j \leq |W|-1}$ is also a complete system of matrix units for the central factor $h_W \calB$ of $\calB$, so there is an isomorphism $h_W \calB \cong M_{|W|}(K)$ given by
$$h_W b \mapsto \sum_{i,j=0}^{|W|-1} b_{ij} e_{ij}, \quad \text{with } b_{ij} = e_{1i}(W) \cdot b \cdot e_{j1}(W) \in e_{11}(W) \calB e_{11}(W) \cong K$$
which is also a $*$-isomorphism. In fact, one should note that this $*$-isomorphism coincides with the $*$-isomorphism given in \eqref{equation-*isomorphism}, i.e. $h_W \calB = \calB \chi_W \calB$.

It is now straightforward to see that each $\chi_{T^l(W)}$ is a minimal projection in $\calB$.
\end{proof}

As a consequence of Proposition \ref{proposition-minimal.central} we obtain a $*$-homomorphism from the algebra $\calB$ into an infinite matrix product $\gotR := \gotR(E,\calP) = \prod_{W \in \V} M_{|W|}(K)$ given by
$$\pi : \calB \ra \gotR, \quad \pi(b) = (h_W \cdot b)_W.$$

We will show below that this homomorphism is injective, but for that we need a preliminary lemma.

\begin{lemma}\label{lemma-orthogonal.sum}
Suppose that $b \in \calB_0$, $b \neq 0$ can be written as a finite linear combination of the form
$$b = \sum_{U \in \calU} \lambda_U \chi_U,$$
where the $U \in \calU$ are nonempty pairwise disjoint clopen subsets of $X$, and $\lambda_U \in K^*$. Then there exists a $W \in \V$ such that $h_W \cdot b \neq 0$.
\end{lemma}
\begin{proof}
Fix one $U \in \calU$. Since $\mu$ is a full measure, $\mu(U) > 0$. Also by Lemma \ref{lemma-quasi.partition}, there exists a $W \in \V$ of length $k \geq 1$ such that $U \cap T^l(W) \neq \emptyset$ for some $0 \leq l \leq k-1$. But then
$$h_W \cdot \chi_{U \cap T^l(W)}b = \lambda_U \chi_{U \cap T^l(W)} \neq 0.$$
It follows that $h_W \cdot b \neq 0$.
\end{proof}

We are now ready to prove injectivity of $\pi$.

\begin{proposition}\label{proposition-embedding.B.R}
With the above hypothesis and notation, we have that the map $\pi : \calB \ra \gotR$ is injective. Moreover, the socle of $\calB$ is essential and coincides with the ideal generated by $\chi_W$, for $W \in \V$, that is
$$\soc(\calB) = \bigoplus_{W \in \V} \calB \chi_W \calB \cong \bigoplus_{W \in \V} M_{|W|}(K).$$ 
\end{proposition}

\begin{proof}
For injectivity, it is enough to show that the ideal $\bigoplus_{W \in \V} h_W \calB$ is essential in $\calB$ or, equivalently, that for any nonzero element $b \in \calB$, we can always find a $W \in \V$ such that $h_W \cdot b \neq 0$. By writing $b$ as a finite sum
$$b = \sum_{i = -n}^{m} b_i (\chi_{X \backslash E}t)^i = \sum_{i = -n}^{m} b_i t^i$$
with each $b_i \in \calB_i$ and $b_{-n} \neq 0$ (where $n \in \Z$), it is enough to show that there exists a $W \in \V$ such that $h_W \cdot b_{-n} \neq 0$. But this follows immediately from Lemmas \ref{lemma-structure.bi} and \ref{lemma-orthogonal.sum}. We obtain that $\pi$ is injective, and also that the ideal $\bigoplus_{W \in \V} h_W \calB$ is essential in $\calB$.


Now since each $\chi_W$ is a minimal projection by Proposition \ref{proposition-minimal.central}, it follows that the ideal of $\calB$ generated by $\chi_W$ is contained in the socle of $\calB$. This says that $\bigoplus_{W \in \V} \calB \chi_W \calB \subseteq \soc(\calB)$. In particular, this shows that $\soc(\calB)$ is essential in $\calB$, and from the general fact that the socle is contained in any essential ideal we conclude that $\bigoplus_{W \in \V} \calB \chi_W \calB = \soc(\calB)$, as required.
\end{proof}

Given a monomial $b_i t^i$ (resp. $b_{-j} t^{-j}$) with $b_i \in \calB_i$ (resp. $b_{-j} \in \calB_{-j}$), we are interested in computing its image under $\pi$. By Lemma \ref{lemma-comm.B0}, we can write $b_i$ (resp. $b_j$) as a linear combination of characteristic functions of nonempty sets of the form
\begin{equation}\label{equation-form.of.bs}
T^{s-1}(Z'_{s-1}) \cap T^{s-2}(Z'_{s-2})\cap \cdots \cap Z'_0 \cap T^{-1}(Z'_{-1}) \cap \cdots \cap T^{-r}(Z'_{-r}), 
\end{equation}
where $r,s \geq 0$. Note that since $b_i = \chi_{X \setminus (E\cup T(E) \cup \cdots \cup T^{i-1}(E))} b_i$ (resp. $b_{-j} = \chi_{X\setminus (T^{-1}(E)\cup \cdots \cup T^{-j}(E))} b_{-j}$), we can (and will) assume, by expanding these sets if necessary, that $s \geq i$ (resp. $r \geq j$).

\begin{definition}\label{definition-occurrence}
Assume that $W$ is in standard form \eqref{equation-Wexpression}, i.e.
$$W = E \cap T^{-1}(Z_1) \cap \cdots \cap T^{-k+1}(Z_{k-1}) \cap T^{-k}(E)$$
for some $k \geq 1$ and some $Z_1,...,Z_{k-1} \in \calP$. We say that a sequence $(Z'_{s-1},\dots , Z_0',\dots , Z'_{-r})$ of elements of $\calP$ \textit{occurs} in $W$ if there exists $l \geq 0$ such that
$$Z_{l+1}= Z'_{s-1},\quad Z_{l+2}= Z'_{s-2}, \quad \dots, \quad Z_{l+s}= Z'_0, \quad Z_{l+s+1}= Z'_{-1},\quad  \dots ,\quad Z_{l+s+r}= Z'_{-r}.$$
That is, if the sequence $(Z'_{s-1},\dots , Z_0',\dots , Z'_{-r})$ occurs as a subsequence of $(Z_1,Z_2,...,Z_{k-1})$ displaced $l$ positions to the right. In this case we say that $l$ is an \textit{occurrence} of $(Z'_{s-1},\dots , Z_0',\dots , Z'_{-r})$ in $W$. Note that a necessary condition for the sequence $(Z'_{s-1},\dots , Z_0',\dots , Z'_{-r})$ to be an occurrence is that $s+r \leq k-1$.
\end{definition}

Observe that, by definition, if we let $S = T^{s-1}(Z'_{s-1}) \cap \cdots \cap Z'_0 \cap \cdots \cap T^{-r}(Z'_{-r})$ be given by \eqref{equation-form.of.bs}, then $l$ is an occurrence of $(Z'_{s-1},...,Z'_0,...,Z'_{-r})$ in $W$ if and only if $s+r \leq k-1$ and $T^{l+s}(W) \cap S$ is nonempty, and in this case necessarily $T^{l+s}(W) \cap S = T^{l+s}(W)$.

\begin{lemma}\label{lemma-image.of.pi}
Assume the above notation and that $i,j \geq 0$. We have:
\begin{enumerate}[i)]
\item If $b_i = \chi_S$, $S$ of the form \eqref{equation-form.of.bs}, then $h_W \cdot b_it^i $ is nonzero if and only if $(Z'_{s-1},\dots , Z_0',\dots , Z'_{-r})$ occurs in $W$, and in this case we have
 $$h_W \cdot b_it^i = \sum_l e_{l+s,l+s-i}(W) ,$$
 where $l$ ranges over the set of occurrences of $(Z'_{s-1},\dots , Z_0',\dots , Z'_{-r})$ in $W$. 
  \item Suppose that $b_{-j} = \chi_S$ with $S$ also given by \eqref{equation-form.of.bs}. Then $h_W \cdot b_{-j}t^{-j}$ is nonzero if and only if $(Z'_{s-1},\dots , Z_0',\dots , Z'_{-r})$ occurs in $W$, and in this case we have
 $$h_W \cdot b_{-j}t^{-j} = \sum_l e_{l+s,l+s+j}(W) ,$$
 where $l$ ranges over the set of occurrences of $(Z'_{s-1},\dots , Z_0',\dots , Z'_{-r})$ in $W$.
\end{enumerate}
Observe that in any case the formula is valid, that is
$$h_W \cdot b_it^i = \sum_l e_{l+s,l+s-i}(W) \quad \text{ whenever } i \in \Z,$$
where $l$ ranges over the set of occurrences of $(Z'_{s-1},\dots , Z_0',\dots , Z'_{-r})$ in $W$.
\end{lemma}

\begin{proof}
$i)$ This is a simple computation. Write $b_i = \chi_S$ with $S = T^{s-1}(Z'_{s-1}) \cap \cdots \cap Z'_0 \cap \cdots \cap T^{-r}(Z'_{-r})$. We compute $h_W \cdot b_i t^i = \sum_{l=0}^{k-1} \chi_{T^l(W) \cap S} t^i$. By the observation preceding the lemma this sum equals $\sum_{l} \chi_{T^{l+s}(W)} t^i$, where $l$ ranges over the set of occurrences of $(Z'_{s-1},\dots , Z_0',\dots , Z'_{-r})$ in $W$. Since $s \geq i$,
\begin{align*}
\chi_{T^{l+s}(W)} & t^i = (\chi_{X \backslash E} t)^{l+s} \chi_W (t^{-1} \chi_{X \backslash E})^{l+s} t^i = t^{l+s} \chi_{X \backslash (T^{-1}(E) \cup \cdots \cup T^{-l-s}(E))} \cdot \chi_W t^{-l-s} t^i \\
& = (\chi_{X \backslash E} t)^{l+s} \chi_W (t^{-1} \chi_{X \backslash E})^{l+s-i} = e_{l+s,l+s-i}(W).
\end{align*}
The result follows from this computation. The proof of $ii)$ is similar. 
\end{proof}

There are two relatively special elements inside $\calB$ that we are interested in computing their images under $\pi$ for later use (see Lemma \ref{lemma-compat.measure}). These are the elements $\chi_E = 1 - (\chi_{X \backslash E}t)(\chi_{X \backslash E}t)^*$ and $\chi_{T^{-1}(E)} = 1 - (\chi_{X \backslash E}t)^*(\chi_{X \backslash E}t)$. Their images under $\pi$ are easy to compute: for $W \in \V$, we have $\chi_W \cdot \chi_E = \chi_W$ and $\chi_{T^l(W)} \cdot \chi_E = 0$ for $1 \leq l \leq |W|-1$, so $h_W \cdot \chi_E = \chi_W = e_{00}(W)$ and
$$\pi(\chi_E) = (e_{00}(W))_W.$$

We also have $\chi_{T^l(W)} \cdot \chi_{T^{-1}(E)} = 0$ for $0 \leq l \leq |W| - 2$ and $\chi_{T^{|W|-1}(W)} \cdot \chi_{T^{-1}(E)} = \chi_{T^{|W|-1}(W)}$, so $h_W \cdot \chi_{T^{-1}(E)} = \chi_{T^{|W|-1}(W)} = e_{|W|-1,|W|-1}(W)$ and
$$\pi(\chi_{T^{-1}(E)}) = (e_{|W|-1,|W|-1}(W))_W.$$

\subsection{Example: the lamplighter group algebra}\label{subsection-lamplighter.group}

We now show that the so-called lamplighter group algebra $K[\Gamma]$ can be realized as a crossed product algebra of the above form, so we can apply our main construction to $K[\Gamma]$. In a future paper \cite{AC} such a construction will be used in order to study possible $l^2$-Betti numbers arising from $K[\Gamma]$.

The lamplighter group $\Gamma$ is defined to be the wreath product of the finite group of two elements $\Z_2$ by $\Z$. 
In other words, $\Gamma = \Z_2 \wr \Z = \Big( \bigoplus_{i \in \Z} \Z_2 \Big) \rtimes_{\sigma} \Z$, where $\sigma : \Z \act H = \bigoplus_{i \in \Z} \Z_2$ is the Bernoulli shift defined by $\sigma(n)(x)_i = x_{i+n}, x = (x_i)_i \in H$. In terms of generators and relations, $\Gamma$ is generated by $\{a_i\}_{i \in \Z}$ and $t$, satisfying the following relations (here $1$ denotes its unit element):
$$a_i^2=1, \quad a_ia_j = a_ja_i, \quad ta_it^{-1} = a_{i-1} \quad \text{ for all }i,j \in \Z.$$
Let now $X = \wh{H}$ be the Pontryagin dual of $H$, which we identify with the Cantor set $\prod_{i \in \Z}\{0,1\}$, and let $T : X \ra X$ be the shift homeomorphism, namely $T(x)_i = x_{i+1}$ for $x \in X$. Take $K$ to be a field of characteristic different from $2$. Fourier 
transform (or Pontryagin duality) gives a $*$-isomorphism
$$\mathcal F : K[\Gamma] \ra C_K(X) \rtimes_T \Z$$
where $t$ is mapped to the generator of $\Z$, also denoted by $t$, and $a_i$ is mapped to $\chi_{U_i} - \chi_{\ol{U_i}}$, being $U_i = \{x \in X \mid x_i = 0\}$ and $\ol{U_i}$ its complement in $X$. In particular, the elements $e_i = \frac{1+a_i}{2}$ are idempotents in $K[\Gamma]$, which correspond to the characteristic function of $U_i$.

\begin{notation}\label{notation-cylinder.sets}
Given $\varepsilon_{-k},...,\varepsilon_l \in \{0,1\}$, the cylinder set $\{x \in X \mid x_{-k} = \varepsilon_{-k},...,x_l = \varepsilon_l\}$ will be denoted by $[\varepsilon_{-k} \cdots \ul{\varepsilon_0} \cdots \varepsilon_l]$. So for example $\calU_0 = [\ul{0}]$, and its characteristic function $\chi_{[\ul{0}]}$ is identified with $e_0$ under $\mathcal F$.
\end{notation}

We have a natural measure $\mu$ on $X$ given by the usual product measure, where we take the $\big(\frac{1}{2},\frac{1}{2}\big)$-measure on each component $\{0,1\}$. It is well-known (cf. \cite[Example 3.1]{KM}) that $\mu$ is an ergodic, full and shift-invariant probability measure on $X$. Therefore we can apply our methods from Section \ref{subsection-quasi.partitions.representation.B}. For a fixed $n \geq 1$, we take $E = [1...\ul{1}...1]$ (with $2n+1$ one's), and the partition of the complement $\calP$ given by the obvious one, namely
$$\calP = \{ [00...\ul{0}...00],[00...\ul{0}...01],...,[01...\ul{1}...11]\}.$$
Here $\calB$ coincides with the unital $*$-subalgebra of $K[\Gamma]$ generated by the partial isometries $s_i = e_it$ for $i=-n,...,n$.

Also, the quasi-partition $\ol{\calP}$ consists here of the translates of the sets $W \in \V$ of the following form:
\begin{enumerate}[a)]
\item $W_0 = E \cap T^{-1}(E) = [11...\ul{1}...111]$ of length $1$ (there are $2n+2$ one's);
\item $W_1 = [11...\ul{1}...11011...{1}...11]$ of length $2n+2$ (there are $4n+2$ one's, and a zero);
\item for $l \geq 0$, $W(\varepsilon_1,...,\varepsilon_l) = [11...\ul{1}...110\varepsilon_1 \cdots \varepsilon_l 011...{1}...11]$ of length $2n+3+l$, where $(\varepsilon_1,...,\varepsilon_l) \in \{0,1\}^l$ is a sequence having 
at most $2n$ consecutive one's.
\end{enumerate}
We refer the reader to \cite{AC} for more details. It is worth to mention that the $*$-algebra $\calB$ corresponding to $n=0$ is the $*$-algebra considered in \cite{AG}, concretely it is 
$*$-isomorphic to the semigroup algebra $K[\calF]$ of the monogenic free inverse monoid $\calF$ (see \cite[Section 4 and Proposition 6.5]{AG}).

\section{Sylvester matrix rank functions on \texorpdfstring{$\calA$}{}}\label{section-Sylvester.rank.functions.A}

Throughout this section, $T$ will denote a homeomorphism of an infinite, totally disconnected, compact metrizable space $X$, and $\mu$ will denote a full ergodic $T$-invariant Borel probability measure on $X$. Note that this implies that $\mu$ is atomless, that is,  $\mu (\{ x \})=0$ for all $x\in X$.

\subsection{Approximation algebras}\label{subsection-approximation.algebras}

We make our construction from Section \ref{section-main.construction} to depend on a point $y \in X$. Let $\{E_n\}_{n \geq 1}$ be a decreasing sequence of clopen sets of $X$ such that $\bigcap_{n \geq 1} E_n= \{ y \}$, and let $\calP_n$ be partitions of $X \backslash E_n$ such that $\calP_{n+1} \cup \{E_{n+1}\}$ is finer than $\calP_n \cup \{E_n\}$; so $E_n$ is the disjoint union of $E_{n+1}$ and some of the sets in $\calP_{n+1}$.

\begin{hypothesis}\label{hypothesis-generate.top}
We also require that $\bigcup_{n \geq 1} (\calP _n \cup \{E_n\})$ generates the topology of $X$.
\end{hypothesis}

Recalling Lemma \ref{lemma-quasi.partition}, each quasi-partition $\ol{\calP}_n$ consists of all the $T$-translates of the nonempty subsets of $X$ of the form
$$W = E_n \cap T^{-1}(Z_1) \cap \cdots \cap T^{-k+1}(Z_{k-1}) \cap T^{-k}(E_n)$$
for some $k \geq 1$ and some $Z_1,...,Z_{k-1} \in \calP_n$. We write $\V_n$ for the set of all the $W \in \ol{\calP}_n$ of the above form. We thus have $\ol{\calP}_n = \{ T^l(W) \mid W \in \V_n, 0\le l \le |W|-1 \}$ for all $n$.

In these conditions, it follows that the quasi-partition $\ol{\calP}_{n+1}$ constructed from the partition $\calP_{n+1} \cup \{E_{n+1}\}$ is finer than 
the quasi-partition $\ol{\calP}_n$ constructed from the partition $\calP_n \cup \{E_n\}$. Indeed, let $W' \in \V_{n+1}$ and write it as
$$W' = E_{n+1} \cap T^{-1}(Z'_1) \cap \cdots \cap T^{-k+1}(Z'_{k-1}) \cap T^{-k}(E_{n+1})$$
for $k\ge 1$ and $Z'_1, ..., Z'_{k-1} \in \calP_{n+1}$. Since $\calP_{n+1} \cup \{ E_{n+1}\}$ is finer than $\calP_n \cup \{E_n\}$, there exist unique integers $1 \leq k_1 < \cdots < k_r < k$ and unique elements $Z_j \in \calP_n$ for $j \in \{1,...,k-1\} \backslash \{k_1,...,k_r\}$ such that
\begin{equation*}
\begin{split}
	Z'_1 &\subseteq Z_1, \\
	Z'_2 &\subseteq Z_2, \\
	& \vdots \\
	Z'_{k_1-1} &\subseteq Z_{k_1-1},\\
	\text{ and } Z&'_{k_1} \subseteq E_n;
\end{split}\qquad
\begin{split}
	Z'_{k_1+1} &\subseteq Z_{k_1+1}, \\
	Z'_{k_1+2} &\subseteq Z_{k_1+2}, \\
	& \vdots \\
	Z'_{k_2-1} &\subseteq Z_{k_2-1},\\
	\text{ and } Z&'_{k_2} \subseteq E_n;
\end{split}\qquad \cdots \qquad
\begin{split}
	Z'_{k_r+1} &\subseteq Z_{k_r+1}, \\
	Z'_{k_r+2} &\subseteq Z_{k_r+2}, \\
	& \vdots \\
	Z'_{k-1} &\subseteq Z_{k-1}.
\end{split}
\end{equation*}
Therefore
\begin{equation}\label{equation-quasi.partition.finer}
\begin{aligned}
W' \subseteq & \Big( E_n \cap T^{-1}(Z_1) \cap \cdots \cap T^{-k_1+1}(Z_{k_1-1}) \cap T^{-k_1}(E_n) \Big) \\
& \cap T^{-k_1}\Big( E_n \cap T^{-1}(Z_{k_1+1}) \cap \cdots \cap T^{-k_2+k_1+1}(Z_{k_2-1}) \cap T^{-k_2+k_1}(E_n) \Big) \\
& \cap \cdots  \cap T^{-k_r}\Big( E_n \cap T^{-1}(Z_{k_r+1}) \cap \cdots \cap T^{-k+k_r+1}(Z_{k-1}) \cap T^{-k+k_r}(E_n) \Big) \\
= & W_0 \cap T^{-k_1}(W_1) \cap \cdots \cap T^{-k_r}(W_r),
\end{aligned}
\end{equation}
where each $W_i = E_n \cap T^{-1}(Z_{k_i+1}) \cap \cdots \cap T^{-k_{i+1}+k_i+1}(Z_{k_{i+1}-1}) \cap T^{-k_{i+1}+k_i}(E_n)$ belongs to $\V_n$, and that they are not necessarily distinct. From here, it is clear that for $0 \leq l \leq k-1$, $T^l(W')$ is contained in some translate of some $W_i$, and so $\ol{\calP}_{n+1}$ is finer than $\ol{\calP}_n$.

\text{ }\vspace{-0.2cm}

In this way, we construct a sequence of approximating algebras $\calA_n := \calA(E_n, \calP_n)$ (see Definition \ref{definition-approx.alg}) 
such that $\calA_n \subseteq \calA_{n+1}$, where the inclusions are given by the embeddings $\iota_n(\chi_Z \cdot t) = \sum_{Z'} \chi_{Z'} \cdot t$, where the sum is over all 
the $Z' \in \calP_{n+1}$ contained in $Z$. By Proposition \ref{proposition-embedding.B.R}, we have embeddings $\pi_n : \calA_n \ra \gotR_n$ where $\gotR_n = \prod_{W \in \V_n} M_{|W|}(K)$, given by $\pi_n(a) = (h_W \cdot a)_W$.

We can build a generalized Bratteli diagram associated to such construction, such that each vertex receives a finite number of edges, and we can order this set of edges in the same way as for the case of an essentially minimal homeomorphism, see for instance \cite{HPS92} for the latter. The only difference is that there are a possibly infinite number of vertices at each level, and that these vertices might emit in principle an infinite number of edges. This can be done as follows.

The vertices at the level $n$ of this generalized Bratteli diagram are the sets $W \in \V_n$, that is, the sets
\begin{equation}\label{equation-sets.for.Pn}
W = E_n \cap T^{-1}(Z_1)\cap \cdots \cap T^{-k+1}(Z_{k-1}) \cap T^{-k}(E_n),
 \end{equation}
where $Z_1,...,Z_{k-1} \in \calP _n$. There is an arrow from a vertex $W \in \V_n$ to a vertex $W' \in \V_{n+1}$ if $W$ appears as a segment of the sequence corresponding to $W'$; more precisely, if $W$ equals to some $W_i$, being
$$W' \subseteq W_0 \cap T^{-k_1}(W_1) \cap \cdots \cap T^{-k_r}(W_r)$$
as in \eqref{equation-quasi.partition.finer}. Equivalently, if $W' \subseteq T^{-j'}(W)$ or even $W' \cap T^{-j'}(W) \neq \emptyset$ for some $0 \leq j' < |W'|-1$. If no such $j'$ exists, then there are no arrows from $W$ to $W'$. The edges ending at $W'$ are linearly ordered according to the integers $j'$.

Thus, the set of arrows $W \ra W'$ is in bijective correspondence with the set $J(W,W') = \{0 \leq j' < |W'|-1 \mid W' \subseteq T^{-j'}(W)\}$. Clearly $J(W,W')$ is always a finite set, and each $W'$ receives at least one arrow.

\begin{proposition}\label{proposition-embed.Rn}
Following the above notation, we can embed each $\gotR_n$ into $\gotR_{n+1}$ via the construction of the generalized Bratteli diagram just mentioned.
\end{proposition}
\begin{proof}
Recall that $\gotR_n = \prod_{W \in \V_n} M_{|W|}(K)$, $\gotR_{n+1} = \prod_{W' \in \V_{n+1}} M_{|W'|}(K)$. Since each $W'$ receives a finite number of arrows in the diagram, it will be sufficient to define the connecting maps $j_n : \gotR_n \ra \gotR_{n+1}$ on each simple factor $\varphi_W : M_{|W|}(K) \ra \gotR_{n+1}$, because in this case each $j_n$ will be defined as
$$j_n : \gotR_n \ra \gotR_{n+1}, \quad (a_W)_W \mapsto \sum_W \varphi_W(a_W).$$
We define $\varphi_W$ to be the block diagonal $*$-homomorphism
$$\varphi_W(e_{ij}(W)) := \Big(\sum_{j' \in J(W,W')} e_{i+j',j+j'}(W')\Big)_{W'}.$$
Since every $W \in \V_n$ always emits at least one arrow to some $W' \in \V_{n+1}$ (this is clear since the translates of the $W'$ form a quasi-partition of $X$), the maps $\varphi_W$ are injective, and so $j_n$ is an embedding of $*$-algebras.
\end{proof}

By construction, we obtain commutative diagrams
\begin{equation}
\vcenter{
	\xymatrix{
	\cdots \ar[r] & \calA_n \ar[r]^{\iota_n} \ar[d]^{\pi_n} & \calA_{n+1} \ar[r]^{\iota_{n+1}} \ar[d]^{\pi_{n+1}} & \calA_{n+2} \ar[d]^{\pi_{n+2}} \ar[r] & \cdots \\
	\cdots \ar[r] & \gotR_n \ar[r]^{j_n} & \gotR_{n+1} \ar[r]^{j_{n+1}} & \gotR_{n+2} \ar[r] & \cdots
	}
}\label{figure-comm.diag.1}
\end{equation}

Set now $\gotR_{\infty} = \varinjlim_n (\gotR_n,j_ n)$ and $\calA_{\infty} = \varinjlim_n (\calA_n, \iota_n) = \bigcup_{n \geq 1} \calA_n $. Note that each $\gotR_n$ is a regular ring, and so is its inductive limit $\gotR_{\infty}$. Moreover, by the commutativity of the diagrams \eqref{figure-comm.diag.1} and the fact that each $\pi_n$ is injective, the algebra $\calA _{\infty}$ is obviously a $*$-subalgebra of $\gotR_{\infty}$, through the limit map $\pi_{\infty} : \calA_{\infty} \ra \gotR_{\infty}$.

A description of the algebra $\calA_{\infty}$ in terms of the crossed product is given as follows. For an open set $U$ of $X$, we denote by $C_{c,K}(U)$ the ideal of $C_K(X)$ generated by the characteristic functions $\chi_V$, where $V$ is a clopen subset of $X$ such that $V \subseteq U$.

\begin{lemma}\label{lemma-descr.Ainfty}
Let $\calA_y$ be the $*$-subalgebra of $\calA= C_K(X) \rtimes_T \Z$ generated by $C_K(X)$ and $C_{c,K}(X \backslash \{ y \})t$. Then we have $\calA_{\infty} = \calA_y$. 
\end{lemma}

\begin{proof} 
$\calA_{\infty}$ is generated, as a $*$-algebra, by $1 = \chi_X \in \calA$ and the partial isometries $\chi_Z t$ for every $Z \in \bigcup_{n \geq 1} \calP_n$. It is then clear that $\calA_{\infty} \subseteq \calA _y$ because $1 = \chi_X \in C_K(X)$ and $\chi_Z t \in C_{c,K}(X \backslash \{ y \})t$ for every $Z \in \bigcup_{n \geq 1} \calP _n$. 
  
For the other inclusion, we first check that $C_{c,K}(X \backslash \{ y \})t \subseteq \calA_{\infty}$, so let $C$ be a clopen subset of $X$ such that $y \notin C$. Since $C$ is closed and $y \notin C$, there exists an index $n_0 \geq 1$ such that $E_{n_0} \cap C = \emptyset$, and so $E_n \cap C = \emptyset$ for $n \geq n_0$ (because $E_n \subseteq E_{n_0}$ for $n \geq n_0$). Since $C$ is also open and $\bigcup_{n \geq n_0} (\calP_n \cup \{ E_n \})$ generate the topology of $X$, we can write $C = \bigcup_{i \geq 1} Z_i$ for $Z_i \in \bigcup_{n \geq n_0} \calP_n$. But $C$ is also compact, so this countable union is in fact finite, and we can further assume without loss of generality that the $Z_i's$ all belong to the same $\calP_N$, and thus are pairwise disjoint. Therefore $C = \bigsqcup_{i=1}^s Z_i$. But now we get that
$$\chi_C t = \sum_{i=1}^s \chi_{Z_i}t \in \calA_N \subseteq \calA_{\infty}.$$
This shows that $C_{c,K}(X \backslash \{ y \})t \subseteq \calA_{\infty}$. Next, we show that $C_K(X) \subseteq \calA_{\infty}$. Indeed, if $C$ is a clopen subset of $X$ and $y \notin C$, then the above argument gives that $\chi_C = (\chi_Ct)(\chi_Ct)^*$ belongs to $\calA_{\infty}$. If $y \in C$ then $\chi_C = 1 - \chi_{X \backslash C}\in \calA_{\infty}$. This concludes the proof.
\end{proof}

\begin{remark}\label{remark-minimal.Cantor.systems}
An analogue of the algebra $\calA_y$ appears in the theory of minimal Cantor systems, see e.g. \cite{P89}, \cite{HPS92}, \cite{GPS}. Let $(X,\varphi)$ be a minimal Cantor system and take $y \in X$. In these papers, the $C^*$-subalgebra $A_y$ of the $C^*$-crossed product $A= C(X) \rtimes_{\varphi} \Z$ which is generated by $C(X)$ and $C(X \backslash \{ y \})u$, where $u$ is the canonical unitary in the crossed product implementing $\varphi$, is considered, and it is shown that $A_y$ is an AF-algebra.

Although our algebra $\calA _y$ is (in general) not ultramatricial, we have shown in Lemma \ref{lemma-descr.Ainfty} that $\calA_y = \calA_{\infty}$, a direct limit of algebras $\calA_n$ which are subalgebras of infinite products of matrix algebras over $K$, which can be considered as a replacement of being just finite products of full matrix algebras over $K$.
\end{remark}

We may determine how big is the subalgebra $\calA_y = \calA_{\infty}$ inside the algebra $\calA$ in some cases of interest.
 
\begin{proposition}\label{proposition-periodic.point}
Let us assume the above notation. Suppose that $y$ is a periodic point for $T$ with period $l$. Let $I$ be the ideal of $\calA$ generated by $C_{c,K} (X \backslash \{y, T(y),\dots T^{l-1}(y)\})$. Then:
\begin{enumerate}[(i)]
\item $I$ is also an ideal of $\calA_{\infty}$, and we have $*$-algebra isomorphisms
$$\calA / I \xra{\cong} M_l(K[s,s^{-1}]), \qquad \calA_{\infty} / I \xra{\cong} M_l(K).$$
\item There exists some $M \geq 0$ such that for each $n \geq M$ there is exactly one $W_n \in \V_n$ of length $l$ and containing $y$, and such that the isomorphism $h_{W_n} \calA_n \cong M_l(K)$ given during the proof of Proposition \ref{proposition-minimal.central} coincides with the restriction of the projection map $q : \calA_{\infty} \ra \calA_{\infty}/I$ on $h_{W_n} \calA_n$, that is, the following diagram commutes.
\begin{equation}
\vcenter{
	\xymatrix{
	\calA_{\infty} \ar[r]^{q} & \calA_{\infty}/I \ar[d]^{\cong} \\
	h_{W_n} \calA_n \ar[r]^{\cong} \ar@{^{(}->}[u] & M_l(K)
	}
}\label{figure-comm.diag.2}
\end{equation}
Moreover, $h_W \in I$ for all $W \in \V_n$, $W \neq W_n$, which means that $h_W$ is the zero matrix under the composition $\calA_{\infty} \ra \calA_{\infty}/I \cong M_l(K)$.
\item $\calA_n / (I \cap \calA_n) \cong \calA_{\infty} / I \cong M_l(K)$ and $(1-h_{W_n}) \calA_n = I \cap \calA_n$ for every $n \geq M$.
\end{enumerate}
\end{proposition}

\begin{proof}
$(i)$ It is clear that $I \subseteq \calA_{\infty}$ because the set $X \backslash \{y,T(y),\dots,T^{l-1}(y)\}$ is an invariant open subset of $X$ and so all elements of $I$ are of the form $\sum_{i=-m}^n f_it^i$ where $f_i \in C_{c,K}(X \backslash \{y, T(y),\dots, T^{l-1}(y)\}) \subseteq C_K(X) \subseteq \calA_y = \calA_{\infty}$; hence if $U_i$ denotes the support of $f_i$, which is a clopen subset of $X$, then
$$f_it^i = f_i \cdot \chi_{U_i} t^i = f_i (\chi_{U_i} t) (\chi_{T^{-1}(U_i)} t) \cdots (\chi_{T^{-i+1}(U_i)} t) \in \calA_y = \calA_{\infty}.$$

Define a map $\Psi \colon \calA \to M_l(K[t,t^{-1}])$ by sending $f \in C_K(X)$ to the diagonal matrix
$$\text{diag}(f(y),...,f(T^{l-1}(y))),$$
and sending $t$ to the matrix $u = t( \sum_{i=0}^{n-2}e_{i+1,i} + e_{0,n-1})$, where $\{e_{ij}\}_{0 \leq i,j \leq n-1}$ are the canonical matrix units in $M_n(K[t,t^{-1}])$.
It is easily verified that
$$u \Psi(f) u^{-1} = \Psi(T(f))$$
for $f \in C_K(X)$.
It follows from the universal property of the crossed product that there is a unique $*$-homomorphism $\Psi : \calA \ra M_l(K[t,t^{-1}])$ extending the above assignments. For an element $f_n t^n \in \calA$ with $n \geq 0$ we have

\[ \Psi(f_n t^n) =  \left( \begin{array}{@{}c|c@{}}
											\hugezero_{\ol{n} \times (l-\ol{n})} &
											\begin{matrix}
												f_n(y)t^n & & \bigzero \\
													& \ddots & \\
												\bigzero & & f_n(T^{\ol{n}-1}(y)) t^n
											\end{matrix} \\
											\\
											\hline \\
											\begin{matrix}
											f_n(T^{\ol{n}}(y)) t^n & & \bigzero \\
											 & \ddots & \\
											 \bigzero & & f_n(T^{l-1}(y))t^n
											\end{matrix} & \hugezero_{(l-\ol{n}) \times \ol{n}}
											
											\end{array}\right) \]
where $\ol{n}$ denotes the unique integer $0 \leq \ol{n} \leq l-1$ such that $n \equiv \ol{n}$ modulo $l$. We can analogously compute it for $n < 0$. In fact, for an arbitrary element $x = \sum_{n \in \Z} f_n t^n \in \calA$ one can check that, 
for $0 \leq i,j \leq l-1$, the $(i,j)$-component of $\Psi(x)$ is given by
\begin{equation}\label{equation-image.Psi}
(\Psi(x))_{i,j} = \sum_{n \in \Z} f_{nl + (i-j)}(T^i(y)) t^{nl + (i-j)} = \Big( \sum_{n \in \Z} f_{nl + (i-j)}(T^i(y)) t^{nl} \Big) t^{i-j}.
\end{equation}
It follows from this that the kernel of $\Psi$ is precisely the ideal $I$. 
The image of $\Psi$ is given by the subalgebra
$$\calS_l := \{ X \in M_l(K[t,t^{-1}]) \mid X_{ij} \in K[t^l, t^{-l}] t^{i-j} \text{ for all } 0 \leq i,j \leq l-1 \}.$$
That is, each entry is a polynomial in $t, t^{-1}$ of the form $X_{ij} = p_{ij}(t^l, t^{-l}) t^{i-j}$, where $p_{ij}(s,s^{-1}) \in K[s, s^{-1}]$. There is a $*$-isomorphism between $\calS_l$ and the $*$-algebra $M_l(K[s,s^{-1}])$ by defining
$$\ol{\Psi} : \calS_l \ra M_l(K[s,s^{-1}]), \quad X = (X_{ij}) \mapsto Y = (Y_{ij}) \text{ with } Y_{ij} = p_{ij}(s,s^{-1}).$$
Putting everything together, we obtain a $*$-isomorphism $\calA / I \cong \calS_l \cong M_l(K[s,s^{-1}])$, as desired.

Now, we restrict the map $\Psi$ to $\calA_{\infty}$. Since $I \subseteq \calA_{\infty}$, the kernel of this restriction is again $I$, so we only need to study its image. Take $x = \sum_{n \in \Z} f_n t^n \in \calA_{\infty}$, so $x \in \calA_N$ for some $N \geq 1$. We see from the restrictions of the coefficients (see the paragraph just before Observation \ref{observation-approx.t}, also Lemma \ref{lemma-descr.Ainfty}) that $f_n = \chi_{X \backslash (E_N \cup \cdots \cup T^{n-1}(E_N))} f_n$ and $f_{-n} = \chi_{X \backslash (T^{-1}(E_N) \cup \cdots \cup T^{-n}(E_N))} f_{-n}$ for $n \geq 1$, so by \eqref{equation-image.Psi},
\begin{equation*}
 (\Psi(x))_{i,j} = f_{i-j}(T^i(y))t^{i-j}.
\end{equation*}
It follows from this that the image of $\calA_{\infty}$ under the composition $\ol{\Psi} \circ \Psi$ is precisely $M_l(K)$.

$(ii)$ Take $M \geq 0$ such that $T(y), ..., T^{l-1}(y) \notin E_M$, so that $T(y), ..., T^{l-1}(y)\notin E_n$ for $n \geq M$. From now on, fix $n \geq M$. In this case, there are unique sets $Z_1,...,Z_{l-1} \in \calP_n$ such that $T(y) \in Z_1, ..., T^{l-1}(y) \in Z_{l-1}$. Take then
$$W_n := E_n \cap T^{-1}(Z_1) \cap \cdots \cap T^{-l+1}(Z_{l-1}) \cap T^{-l}(E_n),$$
which is nonempty since $y \in W_n$, and $|W_n| = l$. Note that $W_n$ is the unique satisfying these properties.

In order to prove the commutativity of the diagram \eqref{figure-comm.diag.2} it is enough to prove that, for $0 \leq i,j \leq l-1$, the elements $e_{ij}(W_n) \in h_{W_n}\calA_n \subseteq \calA_{\infty}$ correspond to the matrix units $e_{ij}$ under the composition $\calA_{\infty} \ra \calA_{\infty} / I \cong M_l(K)$, but by $(i)$,
$$e_{ij}(W_n) = (\chi_{X \backslash E_n}t)^i \chi_{W_n} (t^{-1} \chi_{X \backslash E_n})^j \stackrel{\Psi}{\longmapsto} t^{i-j} e_{ij} \stackrel{\ol{\Psi}}{\longmapsto} e_{ij}$$
as we wanted to show. Therefore we obtain a $*$-isomorphism $\calA_{\infty}/I \cong h_{W_n} \calA_n$ given by $e_{ij}(W_n) + I \mapsto e_{ij}(W_n)$. For $W \in \V_n$ with $W \neq W_n$, the idempotents $h_W$ and $h_{W_n}$ are orthogonal, and so $h_W$ is the zero matrix in $h_{W_n} \calA_n$ under the previous $*$-isomorphism. That means $h_W \in I$, as required.

$(iii)$ Clearly $1 - h_{W_n} \in \calA_n$. Under $\calA_{\infty}/I \cong h_{W_n} \calA_n \cong M_l(K)$, the element $(1-h_{W_n}) + I$ corresponds to the zero matrix, so $1-h_{W_n} \in I$ too.

Define $I_n = (1-h_{W_n})\calA_n$. From the previous observation, $I_n \subseteq I \cap \calA_n$, and we aim to show the reverse inclusion. 
By the modular law be have
$$I \cap \calA_n = I \cap (I_n \oplus h_{W_n}\calA_n) = I_n \oplus (h_{W_n} \calA_n \cap I).$$
Since $1 \notin I$ and $h_{W_n} \calA_n$ is simple, we deduce that $h_{W_n} \calA_n \cap I = \{0\}$, and thus $I \cap \calA_n = I_n$. The rest follows trivially.\hfill\qedhere
\end{proof}

\subsection{A rank function on \texorpdfstring{$\calA$}{}}\label{subsection-rank.function.A}

We pass to study the possible rank functions that the $*$-algebras $\calA_{\infty}$, $\calA$ can admit. To start this study, we first concentrate our attention on the approximating algebras $\calA_n$ and the embeddings $\pi_n : \calA_n \hookrightarrow \gotR_n$.

We define a rank function $\rk_{\gotR_n}$ on $\gotR_n = \prod_{W \in \V_n} M_{|W|}(K)$ by taking a concrete convex combination of the normalized rank functions $\rk_{|W|} = \frac{\Rk}{|W|}$ 
on the matrix algebras $M_{|W|}(K)$ (here $\Rk$ denotes the usual rank function of matrices $M \in M_{|W|}(K)$). Namely, we take $\alpha_W = |W| \mu (W)$, where $W \in \V _n$, and define
$$\rk_{\gotR_n}(x) = \sum_{W \in \V_n} \alpha_W \rk_{|W|}(x_W)$$
for $x = (x_W)_W \in \gotR_n$. By Lemma \ref{lemma-quasi.partition},
$$\sum_{W \in \V_n} \alpha_W = \sum _{k \geq 1} \sum_{\substack{W \in \V_n \\ |W| = k}} k \mu(W) = \mu(X) = 1,$$
so we get that $\rk_{\gotR_n}$ is indeed a rank function on $\gotR_n$, and a faithful one since $\alpha_W \neq 0$ for all $W \in \V_n$. Moreover, the embeddings $j_n \colon \gotR_n \to \gotR_{n+1}$ are rank-preserving. To show this, we only have to prove that
\begin{equation}\label{equation-W.W'}
\mu(W) = \sum_{W' \in \V_{n+1}} |J(W,W')| \mu (W'),
\end{equation}
since then for $x \in \gotR_n$,
\begin{align*}
\rk_{\gotR_{n+1}}(j_n(x)) & = \sum_{W' \in \V_{n+1}} \sum_{W \in \V_n} \mu(W') |W'| \rk_{|W'|}(\varphi_W(x_W)_{W'}) = \sum_{W \in \V_n} \sum_{W' \in \V_{n+1}} \mu(W') \Rk(\varphi_W(x_W)_{W'}) \\
& = \sum_{W \in \V_n} \Big( \sum_{W' \in \V_{n+1}} \mu(W') |J(W,W')| \Big) \Rk(x_W) = \sum_{W \in \V_n} \mu(W) |W| \rk_{|W|}(x_W) = \rk_{\gotR_n}(x).
\end{align*}
But now suppose that $W$ is as in \eqref{equation-sets.for.Pn}. We can write our space $X$ as $X = \bigsqcup_{W' \in \V_{n+1}} \bigsqcup_{l=0}^{|W'|-1} T^l(W')$ up to a set of measure zero, so by intersecting with $W$ one gets
\begin{equation}
W = \bigsqcup_{W' \in \V_{n+1}} \bigsqcup_{l=0}^{|W'|-1} W \cap T^l(W') = \bigsqcup_{W' \in 
\V_{n+1}} \bigsqcup_{j' \in J(W,W')} T^{j'}(W')
\end{equation}
up to a set of measure $0$. From this the equality \eqref{equation-W.W'} follows by invariance of $\mu$.

With this, we can define a faithful rank function on the inductive limit $\gotR_{\infty} = \varinjlim (\gotR_n, j_n)$ induced from the rank functions $\rk_{\gotR_n}$, which we will denote by $\rk_{\gotR_{\infty}}$. In the next lemma we show that we also have compatibility of our measure $\mu$ and this new rank function $\rk_{\gotR_{\infty}}$.

\begin{lemma}\label{lemma-compat.measure}
Let $\pi_n : \calA_n \ra \gotR_n$ and $\pi_{\infty} : \calA_{\infty} \ra \gotR_{\infty}$ be the canonical inclusions. Then:
\begin{enumerate}[i)]
\item The equality $\mu(Z)= \rk_{\gotR_n}(\pi_n(\chi_Z))$ holds for all $Z \in \calP_n \cup \{ E_n \}$. Moreover,
\item $\mu(U) = \rk_{\gotR_{\infty}}(\pi_{\infty}(\chi_U))$ for all clopen subset $U$ of $X$.
\end{enumerate}
\end{lemma}

\begin{proof}
Let us prove the first formula. For $Z = E_n$, by the computation done at the end of Section \ref{subsection-quasi.partitions.representation.B},
$$\rk_{\gotR_n}(\pi_n(\chi_{E_n})) = \sum_{W \in \V_n} \alpha_W \rk_{|W|}(e_{00}(W)) = \sum_{W \in \V_n} \mu(W) = \mu\Big( \bigsqcup_{W \in \V_n} W \Big) = \mu(E_n),$$
where we have used that the sets $\{W\}_{W \in \V_n}$ form a quasi-partition of $E_n$, see Lemma \ref{lemma-quasi.partition}. For $Z \in \calP_n$, also by Lemma \ref{lemma-quasi.partition} the sets $\{Z \cap \ol{W}\}_{\ol{W} \in \ol{\calP}_n} = \{Z \cap T^l(W)\}_{\substack{W \in \V_n \\ 0 \leq l \leq |W|-1}}$ form a quasi-partition of $Z$. Therefore if $W$ is as in \eqref{equation-sets.for.Pn}, then $h_W \cdot \chi_Z = 
\sum_{j : Z_j = Z} e_{jj}(W)$, so
\begin{align*}
\rk_{\gotR_n}(\pi_n(\chi_Z)) & = \sum_{W \in \V_n} \alpha_W \rk_{|W|}\Big( \sum_{j : Z_j = Z} e_{jj}(W) \Big) = \sum_{W \in \V_n} \mu(W) |\{j \mid Z_j = Z\}| \\
& = \sum_{W \in \V_n} \sum_{l=0}^{|W|-1} \mu(T^l(W) \cap Z) = \mu \Big( \bigsqcup_{W \in \V_n} \bigsqcup_{l=0}^{|W|-1} T^l(W) \cap Z \Big) = \mu(Z).
\end{align*}
As a consequence, $\mu(Z) = \rk_{\gotR_{\infty}}(\pi_{\infty}(\chi_Z))$ for all $Z \in \bigcup_{n \geq 1}(\calP_n \cup \{ E_n \})$. Since $\bigcup_{n \geq 1} (\calP_n \cup \{ E_n \})$ generates the topology of $X$, every clopen subset $U$ of $X$ can be written as a finite (disjoint) union of elements of the partitions $\calP_n \cup \{ E_n\}$, so we get that $\mu(U) = \rk_{\gotR_{\infty}}(\pi_{\infty}(\chi_U))$.
\end{proof}

Using this rank function we will define rank functions over $\calA_{\infty}, \calA$. To this aim, we would like to embed our whole algebra $\calA$ inside $\gotR_{\infty}$, but this is (in general) not possible. What we will do is to embed $\calA$ inside the \textit{rank completion} $\gotR_{\rk}$ of $\gotR_{\infty}$ with respect to its rank function $\rk_{\gotR_{\infty}}$.

From now on we will not write down explicitly the maps $\pi_n, \pi_{\infty}$ and $j_n$, so we will identify

\begin{equation}
\vcenter{
	\xymatrix{
	\calA_n \ar@{^{(}->}[r] \ar@{^{(}->}[d] & \calA_{n+1} \ar@{^{(}->}[d] \ar@{^{(}->}[r] & \calA_{n+2} \ar@{^{(}->}[d] \ar@{^{(}->}[r] & \cdots \ar@{^{(}->}[r] & \calA_{\infty} \ar@{^{(}->}[d] \ar@{^{(}->}[r] & \calA \\
	\gotR_n \ar@{^{(}->}[r] & \gotR_{n+1} \ar@{^{(}->}[r] & \gotR_{n+2} \ar@{^{(}->}[r] & \cdots \ar@{^{(}->}[r] & \gotR_{\infty} \ar@{^{(}->}[r] & \gotR_{\rk}
	}
}\label{figure-comm.diag.3}
\end{equation}
whenever convenient.

\begin{theorem}\label{theorem-completion.A.Ay}
Let $\gotR_{\rk}$ be the rank completion of the regular rank ring $\gotR_{\infty}$ with respect to the rank function $\rk_{\gotR_{\infty}}$. We denote by $\rk_{\gotR_{\rk}} := \ol{\rk_{\gotR_{\infty}}}$ the rank function on $\gotR_{\rk}$ extended from $\rk_{\gotR_{\infty}}$. We then have an embedding
$$\calA \hookrightarrow \gotR_{\rk}$$
that induces a faithful Sylvester matrix rank function, denoted by $\rk_{\calA}$, on $\calA$. In turn, the natural inclusion $\calA_{\infty} \subseteq \calA$ induces a faithful Sylvester matrix rank function, denoted by $\rk_{\calA_{\infty}}$, on $\calA_{\infty}$.

Moreover, we have $\ol{\calA_{\infty}}^{\rk_{\calA_{\infty}}} = \ol{\calA}^{\rk_{\calA}} = \gotR_{\rk}$. 
\end{theorem}

\begin{proof}
The function $\rk_{\gotR_{\infty}}$ is a rank function on $\gotR_{\infty}$, and in fact a faithful Sylvester matrix rank function since $\gotR_{\infty}$ is regular, so there is an embedding 
of $\gotR_{\infty}$ into its completion $\gotR_{\rk}$, which is a regular self-injective rank-complete ring (\cite[Theorem 19.7]{Goo91}). 
This shows that $\calA_{\infty} \hookrightarrow \gotR_{\infty} \subseteq \gotR_{\rk}$, and we will simply identify $\calA_{\infty} \subseteq \gotR_{\infty}$. Now we show that there is a natural embedding of $\calA$ into $\gotR_{\rk}$.

Observe that $\{ \chi_{X \backslash E_n}t\}_{n \in \N}$ is a Cauchy sequence in $\gotR_{\rk}$, because for $n \geq m$ and using Lemma \ref{lemma-compat.measure},
$$\rk_{\gotR_{\rk}}(\chi_{X \backslash E_n}t - \chi_{X \backslash E_m}t) \leq \rk_{\gotR_{\rk}}(\chi_{E_m \backslash E_n}) = \mu(E_m \backslash E_n) \leq \mu(E_m) \xrightarrow[m \ra \infty]{} \mu(\{y\}) = 0.$$
Therefore we may consider the element $u := \lim_n \chi _{X \backslash E_n}t \in \gotR_{\rk}$. It is an invertible element inside $\gotR_{\rk}$ with inverse $u^* = \lim_n t^{-1} \chi _{X \backslash E_n}$, 
since $\lim_n \mu(E_n) = \lim_n \mu(T^{-1}(E_n)) = 0$.
Moreover, the condition $u \chi_C u^{-1} = \chi_{T(C)} = T(\chi_C)$ is easily checked to be true for every clopen subset $C$ of $X$, and so we get that $u f u^{-1} = T(f)$ for every $f \in C_K(X)$. 
It follows from the universal property of the crossed product that there is a unique homomorphism
$$\Phi \colon \calA = C_K(X) \rtimes_T \Z \ra \gotR_{\rk}, \quad \sum_{i \in \Z} f_i t^i \mapsto \sum_{i \in \Z} f_i u^i \text{ }\text{ for } f_i \in C_K(X).$$
This map clearly extends the injective homomorphism $\calA_{\infty} \subseteq \gotR_{\infty} \subseteq \gotR_{\rk}$. To show that it is injective, it suffices to check that $\sum_{i=0}^n f_i u^i $ is never $0$ in $\gotR_{\rk}$ whenever $f_0 \neq 0$ 
and all $f_i \in C_K(X)$. But if $f_0 \ne 0$, and $C$ denotes the support of $f_0$, by taking $s$ big enough so that $n \mu(E_s) < \mu(C)$ we have
\begin{align*}
\rk_{\gotR_{\rk}}(\chi_{X \setminus (E_s \cup \cdots \cup T^{n-1}(E_s))} \cdot \chi_C) & = \mu(X\setminus (E_s \cup \cdots \cup T^{n-1}(E_s) \cup C^c) ) \\
& \geq 1 - (\mu(E_s) + \cdots + \mu(T^{n-1}(E_s)) + \mu(C^c)) = \mu(C) - n \mu(E_s) > 0,
\end{align*}
hence $\chi_{X \setminus (E_s \cup \cdots \cup T^{n-1}(E_s))} f_0 = \chi_{X \setminus (E_s \cup \cdots \cup T^{n-1}(E_s))} \cdot \chi_C f_0 \ne 0$, and moreover
$$\chi_{X \setminus (E_s\cup \cdots \cup T^{n-1}(E_s))}\Big(\sum _{i=0}^n f_i u^i\Big) = \sum_{i=0}^n (\chi_{X \setminus (E_s \cup \cdots \cup T^{n-1}(E_s))} f_i) (\chi_{X\setminus E_s}t)^i \in \calA _s \subseteq \calA_{\infty},$$
and this is nonzero because the map $\calA_{\infty} \subseteq \gotR_{\infty} \subseteq \gotR_{\rk}$ is injective.

We thus get the inclusions $\calA_{\infty} \subseteq \calA \subseteq \gotR_{\rk}$, where we identify $u$ with $t$. Clearly $\rk_{\gotR_{\rk}}$ induce faithful Sylvester matrix rank functions, given by restriction, on either $\calA_{\infty}$ and $\calA$.

For the last part, note that for each $n \geq 1$, $\calA_n$ is dense in $\gotR_n$ with respect to the $\rk_{\gotR_n}$-metric, because 
by Proposition \ref{proposition-embedding.B.R} we have $\text{soc}(\calA_n) = \bigoplus _{W \in \V_n} M_{|W|}(K)$, which is dense in $\gotR_n = \prod_{W \in \V_n} M_{|W|}(K)$. To see this, note that for an element $x \in \gotR_n$, we can consider the sequence of elements $\{x_k\}_{k \geq 1}$ defined by $x_k = \Big(\sum_{\substack{W \in \V_n \\ |W| \leq k}} h_W \Big)x \in \text{soc}(\calA_n)$. A simple computation, using Lemmas \ref{lemma-compat.measure} and \ref{lemma-Rokhlin}, gives
\begin{align*}
\rk_{\gotR_n}(x - x_k) & \leq \rk_{\gotR_n}\Big(1 - \sum_{\substack{W \in \V_n \\ |W| \leq k}} h_W \Big) = \mu \Big( \bigsqcup_{\substack{W \in \V_n \\ |W| > k }} \bigsqcup_{l=0}^{|W|-1} T^l(W) \Big) = \sum_{\substack{W \in \V_n \\ |W| > k}} |W| \mu(W) \xrightarrow[k \ra \infty]{} 0,
\end{align*}
so $x_k \stackrel{k}{\to} x$ in rank. It follows that $\calA_{\infty}$ is dense in $\gotR_{\infty}$, and hence in $\gotR_{\rk}$ with respect to the $\rk_{\gotR_{\rk}}$-metric, so we also get that $\ol{\calA_{\infty}}^{\rk_{\calA_{\infty}}} = \ol{\calA}^{\rk_{\calA}} = \gotR_{\rk}$. 
\end{proof}

It follows that the rank function $\rk_{\gotR_{\rk}}$ on $\gotR_{\rk}$ restricts to a faithful Sylvester matrix rank function on $\calA$ such that $\rk_{\calA}(\chi_U) = \mu(U)$ for each clopen subset $U$ of $X$ (Lemma \ref{lemma-compat.measure}). We now investigate the uniqueness of this rank function, first over $\gotR_{\infty}$ and then over $\calA$ itself.

Given a compact convex set $\Delta$, we denote by $\partial_e \Delta$ the set of extreme points of $\Delta$.

\begin{proposition}\label{proposition-unique.rank}
Following the above notation,
\begin{enumerate}[(i)]
\item the rank function $\rk_{\gotR_{\infty}}$ is a faithful Sylvester matrix rank function on $\gotR_{\infty}$, and it is uniquely determined by the following property: for every clopen subset $U$ of $X$, $\rk_{\gotR_{\infty}}(\pi_{\infty}(\chi_U)) = \mu(U)$.
\item the rank function $\rk_{\calA}$ from Theorem \ref{theorem-completion.A.Ay} is a faithful Sylvester matrix rank function on $\calA$, and it is uniquely determined by the same property as in $(i)$, that is, for every clopen subset $U$ of $X$, $\rk_{\calA}(\chi_U) = \mu(U)$.
\end{enumerate}
Moreover, $\rk_{\gotR_{\infty}} \in \partial_e\Ps(\gotR_{\infty})$ and $\rk_{\calA} \in \partial_e\Ps(\calA)$.
\end{proposition}

\begin{proof}
We first prove $(i)$. As we have already mentioned in the proof of Theorem \ref{theorem-completion.A.Ay}, $\rk_{\gotR_{\infty}}$ is a faithful Sylvester matrix rank function on $\gotR_{\infty}$ because of regularity of the ring. Hence to prove uniqueness of the Sylvester matrix rank function it suffices to check that if $N$ is another Sylvester matrix rank function on $\gotR_{\infty}$ satisfying the required compatibility of the measure, then the restriction of $N$ to $\gotR_{\infty}$ is $\rk_{\gotR_{\infty}}$.

We consider the restriction $N_n$ of $N$ to $\gotR_n$ for every $n \geq 1$, which is a pseudo-rank function on $\gotR_n$ such that $N = \lim_n N_n$. 
Consider any finite subset $S \subseteq \{W \in \V_n\}$. Since $\gotR_n = \prod_{W \in \V_n} M_{|W|}(K)$, we have $\Big( \sum_{W \in S} h_W \Big) \gotR_n = \bigoplus_{W \in S} M_{|W|}(K)$. 
Take the restriction $N_n|_S$ of $N_n$ to $\bigoplus_{W \in S} M_{|W|}(K)$, which turns out to be an unnormalized pseudo-rank function 
on $\bigoplus_{W \in S} M_{|W|}(K)$. Hence it can be written as a combination of the unique normalized rank functions on each simple factor $M_{|W|}(K)$, i.e.
$$N_n|_S = \sum_{W \in S} \beta_W \rk_{|W|} \quad \text{ for some }\beta_W \geq 0 \text{ satisfying } \sum_{W \in S} \beta_W = \sum_{W \in S} N_n(h_W).$$
Now for a fixed $W' \in S$ we compute, by using the compatibility of $N$ with $\mu$,
$$\beta_{W'} = N_n|_S(h_{W'}) = N_n(h_{W'}) = |W'| N_n(\chi_{W'}) = |W'| \mu(W') = \alpha_{W'} .$$
Therefore for an arbitrary element $x \in \gotR_n$, we compute
$$N_n\Big( \Big(\sum_{W \in S}h_W \Big) x \Big) = N_n|_S\Big( \Big(\sum_{W \in S}h_W \Big) x \Big) = \sum_{W \in S} \alpha_W \rk_{|W|}(x_W) = \rk_{\gotR_n}\Big( \Big(\sum_{W \in S}h_W \Big) x \Big).$$
This says that $N_n$ and $\rk_{\gotR_n}$ coincide on $\bigoplus_{W \in S} M_{|W|}(K)$.

Now fix $k \geq 1$, and consider the finite set $S_k = \{W \in \V_n \mid |W| \leq k\}$. For $x \in \gotR_n$, we have the estimate
\begin{align*}
|N_n(x) - \rk_{\gotR_n}(x)| & \leq \Big| N_n(x) - N_n\Big( \Big( \sum_{W \in S_k} h_W \Big) x\Big) \Big| + \Big| \rk_{\gotR_n}\Big( \Big(\sum_{W \in S_k} h_W \Big) x\Big) - \rk_{\gotR_n}(x) \Big| \\
& \leq N_n\Big( 1 - \sum_{W \in S_k} h_W \Big) + \rk_{\gotR_n}\Big( 1 - \sum_{W \in S_k} h_W \Big) \\
& = 2 \mu \Big(  \bigsqcup_{\substack{W \in \V_n \\ |W| > k }} \bigsqcup_{l=0}^{|W|-1} T^l(W) \Big) = 2 \sum_{\substack{W \in \V_n \\ |W| > k}} |W| \mu(W) \xrightarrow[k \ra \infty]{} 0 . 
\end{align*}
Therefore $N_n = \rk_{\gotR_n}$ for all $n \geq 1$, and so $N = \lim_n N_n = \lim_n \rk_{\gotR_n} = \rk_{\gotR_{\infty}}$.

$(ii)$ Let $N$ be a Sylvester matrix rank function on $\calA$ such that $N(\chi_U) = \mu(U)$ for every clopen subset $U$ of $X$. We first check that the restriction $N_n$ of $N$ on $\calA_n$ equals the restriction $\rk_{\calA_n}$ of $\rk_{\calA}$ on $\calA_n$.
Since for any finite subset $S \subseteq \{W \in \V_n\}$ we have the identification $\Big( \sum_{W \in S} h_W \Big) \calA_n = \bigoplus_{W \in S} h_W \calA_n \cong \bigoplus_{W \in S} M_{|W|}(K)$, it follows from the same arguments as in $(i)$ that $N_n(a) = \rk_{\calA_n}(a)$ for every $a \in \calA_n$, and so the restriction $N_{\infty}$ of $N$ on $\calA_{\infty}$ coincides with $\rk_{\calA_{\infty}}$.

Now, to show that $N = \rk_{\calA}$ on $\calA$, it suffices to check that for each algebra generator $a$ of $\calA$ and for each $\varepsilon >0$ there is $b \in \calA_{\infty}$ such that $N(a-b) < \frac{\varepsilon}{2}$ and $\rk_{\calA}(a-b) < \frac{\varepsilon}{2}$. 
This is clear for $a \in C_K(X)$ since $C_K(X) \subseteq \calA_{\infty}$, and it is also clear for $t$, because $\chi_{X\setminus E_n}t \in \calA_n$ and
$$N(t - \chi_{X \backslash E_n}t) \leq N(\chi_{E_n}) = \mu(E_n) \xrightarrow[n \ra \infty]{} 0  , \quad \rk_{\calA}(t - \chi_{X \backslash E_n}t) \leq \rk_{\calA}(\chi_{E_n}) = \mu(E_n) \xrightarrow[n \ra \infty]{} 0.$$
Similarly, we can show that $N$ and $\rk_{\calA}$ coincide on matrices over $\calA$.

Let us show that $\rk_{\calA}$ is extremal. Suppose we have a convex combination $\rk_{\calA} = \alpha N_1 + \beta N_2$, where $N_1$ and $N_2$ are Sylvester matrix rank functions on $\calA$. Assume that $\alpha \neq 0,1$. We first show that each Sylvester matrix rank function $N_i$ induces a $T$-invariant probability measure $\mu_i$ on $X$. For this, we will use an argument similar to the one given in \cite[Lemma 5.1]{RS}. We define premeasures $\ol{\mu}_i$ over the algebra of clopen sets $\K$ of $X$, by the rule
$$\ol{\mu}_i : \K \ra [0,1], \quad \ol{\mu}_i(U) = N_i(\chi_U).$$
By \cite[Theorem 1.14]{Fol84} they can be uniquely extended to measures $\mu_i$ on the Borel $\sigma$-algebra of $X$, and it is straightforward to show that each $\mu_i$ is a $T$-invariant probability measure on $X$.

Now, we necessarily have the equality $\mu = \alpha \mu_1 + \beta \mu_2$, 
and since $\mu$ is extremal (\cite[Theorem 8.1.8]{Phillips}) and $\alpha \neq 0,1$, we obtain that $\mu_1 = \mu_2 = \mu$. This says that $N_i$ are Sylvester matrix rank functions on $\calA$ satisfying $N_i(\chi_U) = \mu_i(U) = \mu(U)$ for each $U \in \K$. By the uniqueness property of part $(ii)$, we get that $N_i = \rk_{\calA}$. It follows that $\rk_{\calA}$ is extremal.

To show that $\rk_{\gotR_{\infty}}$ is extremal, suppose again that we have a convex combination $\rk_{\gotR_{\infty}} = \alpha N_1 + \beta N_2$, where $N_1$ and $N_2$ are pseudo-rank functions on $\gotR_{\infty}$. Assume that $\alpha \neq 0,1$. Then it is clear that each $N_i$ is continuous with respect to $\rk_{\gotR_{\infty}}$, denoted by $N_i << \rk_{\gotR_{\infty}}$, in the sense of \cite[Definition on page 287]{Goo91}, and therefore by \cite[Proposition 19.12]{Goo91}, $N_i$ extend to continuous pseudo-rank functions $\ol{N}_i$ on $\gotR_{\rk}$ such that $\rk_{\gotR_{\rk}} = \alpha \ol{N}_1 + \beta \ol{N}_2$. Since we have an identification $\calA \subseteq \gotR_{\rk}$ given by Theorem \ref{theorem-completion.A.Ay}, the argument above can be used to show that $\rk_{\gotR_{\infty}} \in \partial_e \Ps(\gotR_{\infty})$.
\end{proof}

We can now exactly compute the rank completion $\gotR_{\rk}$ of $\gotR_{\infty}$ (and of $\calA$): it is the well-known von Neumann continuous factor $\calM_K$, which is defined as the completion of $\varinjlim_n M_{2^n}(K)$ with respect to its unique rank function (see \cite{AC2} for details). Moreover, when the involution $*$ on $K$ is positive definite, we can deduce from \cite[Theorem 4.5]{AC2} that there is a $*$-isomorphism between $\gotR_{\rk}$ and $\calM_K$, where the latter has the involution induced from the $*$-transpose involution on each matrix algebra $M_{2^n}(K)$. The above of course applies when $K$ is a subfield of $\C$ which is invariant under complex conjugation. This generalizes a result of Elek \cite{Elek16}.

\begin{theorem}\label{theorem-vN.cont.factor}
There is an isomorphism of algebras $\gotR_{\rk} \cong \calM_K$, the von Neumann continuous factor over $K$. Moreover, if $(K,*)$ is a field with positive definite involution, then $\gotR_{\rk}$ is a $*$-regular ring in a natural way, and $\gotR_{\rk} \cong \calM_K$ as $*$-algebras over $K$.   
\end{theorem}
 
\begin{proof}
Since $\rk_{\gotR_{\infty}}$ is extremal (Proposition \ref{proposition-unique.rank}), it follows from \cite[Theorem 19.14]{Goo91} that $\gotR_{\rk} = \ol{\gotR_{\infty}}^{\rk_{\gotR_{\infty}}}$ is a simple ring. So $\gotR_{\rk}$ is a \textit{continuous factor} in the sense of \cite{AC2}, that is, a simple, (right and left) self-injective regular ring of type $II_f$. Moreover, there is a countably dimensional dense subalgebra of $\gotR_{\rk}$, namely $\calA$, and clearly condition $(iii)$ in \cite[Theorem 2.2]{AC2} is satisfied (because it is satisfied for the dense subalgebra $\gotR_{\infty}$). It follows that $\gotR_{\rk} \cong \calM_K$, the von Neumann continuous factor.

Now assume that $(K,*)$ is a field with positive definite involution. Then each $\gotR_n = \prod_{W \in \V_n} M_{|W|}(K)$ is a $*$-regular ring, where each factor $M_{|W|}(K)$ has the $*$-transpose involution, and the connecting maps $j_n : \gotR_n \ra \gotR_{n+1}$ are given by block-diagonal maps (see Proposition \ref{proposition-embed.Rn}), so in particular are $*$-homomorphisms. Therefore $\gotR_{\infty}$ is a $*$-regular ring, and by \cite[Proposition 1]{Hand77}, the completion $\gotR_{\rk}$ of $\gotR_{\infty}$ is also a $*$-regular ring. One can easily show that $\calA$ sits inside $\gotR_{\rk}$ as a $*$-subalgebra, i.e. that the homomorphism defined in the proof of Proposition \ref{theorem-completion.A.Ay}
preserves the involution.

Now the local condition $(iii)$ in \cite[Theorem 4.5]{AC2} is actually somewhat more difficult to check in this case. Given positive integers $n, k$, if we define $S_k = \{ W \in \V_n \mid |W| \leq k\}$, we have the estimate
$$\rk_{\gotR_{\rk}}\Big( 1 - \sum_{W \in S_k} h_W \Big) = \sum_{\substack{W \in \V_n \\ |W| > k}} |W| \mu(W) \xrightarrow[k \ra \infty]{} 0$$
as we have showed in the proof of Proposition \ref{proposition-unique.rank}. Therefore there exists $k_n$ such that $\rk_{\gotR_{\rk}}(1 - H_n) < \frac{1}{2^n}$, being $H_n$ the projection $\sum_{W \in S_{k_n}} h_W \in \gotR_n$. We approximate $\gotR_{\infty}$ by the unital $*$-subalgebras
$$\gotR_n' := H_n \gotR_n \oplus (1 - H_n)K.$$
Since $H_n \gotR_n = \Big( \sum_{W \in S_{k_n}} h_W \Big) \gotR_n \cong \bigoplus_{W \in S_{k_n}} M_{|W|}(K)$, these algebras are $*$-isomorphic to standard matricial $*$-algebras. Although the sequence of projections $(j_{n,\infty}(H_n))$ is not increasing, there are unital $*$-homomorphisms $j'_{n,m} : \gotR'_n \to \gotR'_m $ for $n\le m$, defined by
$$j_{n,m}' (H_nx +(1-H_n)\lambda ) = H_m \cdot j_{n,m} (H_nx +(1-H_n)\lambda ) + (1-H_m) \lambda ,$$ 
for $x\in \gotR _n$ and $\lambda \in K$. Here $j_{n,m} \colon \gotR_n \ra \gotR_m$ is the natural $*$-homomorphism, and $j_{n,\infty} : \gotR_n \ra \gotR_{\infty}$ the canonical map into the direct limit. Moreover, since each $j_{n,m}$ is given by block-diagonal maps, so are the $j_{n,m}'$. Observe that $(\gotR'_n, j_{n,n+1}')$ is not a directed system, but for $z = H_n x + (1-H_n) \lambda \in \gotR'_n$ and all $m \geq n$, we have the estimate
$$\rk_{\gotR_m}( j_{n,m}(z) - j_{n,m}'(z)) \leq \rk_{\gotR_m}(1-H_m) < \frac{1}{2^m}.$$
Consequently, the proof of the implication $(ii) \implies (iii)$ in \cite[Theorem 4.5]{AC2} can be adapted to the present setting, and we obtain that condition $(iii)$ in \cite[Theorem 4.5]{AC2} holds. This theorem then gives that $\gotR_{\rk} = \ol{\gotR_{\infty}}^{\rk_{\gotR_{\infty}}}$ is $*$-isomorphic to $\calM_K$, as desired.   
\end{proof}

\subsection{Relation with measures on \texorpdfstring{$X$}{}}\label{subsection-relation.measures.X}

Theorem \ref{theorem-completion.A.Ay} and Proposition \ref{proposition-unique.rank} state that, given an ergodic, full and $T$-invariant probability measure $\mu$ on $X$, one can construct an extremal faithful Sylvester matrix rank function $\rk_{\calA}$ on $\calA$, unique with respect to the property that $\rk_{\calA}(\chi_U) = \mu(U)$ for every clopen subset $U$ of $X$. In the next proposition we prove that the converse of this construction can also be made.

\begin{proposition}\label{proposition-rank.measure}
Let $\rk$ be an extremal and faithful Sylvester matrix rank function on $\calA$. Then there exists an ergodic, full and $T$-invariant probability measure $\mu_{\rk}$ on $X$, uniquely determined by the property that
$$\mu_{\rk}(U) = \rk(\chi_U) \quad \text{ for every clopen subset $U$ of $X$}.$$
\end{proposition}
\begin{proof}
It is clear that $\rk$ induces a finitely additive probability measure on the algebra of clopen subsets of $X$ by the rule
$$\ol{\mu}_{\rk}(U) = \rk(\chi_U) \quad \text{ for every clopen subset $U$ of $X$},$$
which, by the same argument as in the proof of Proposition \ref{proposition-unique.rank}, can be uniquely extended to a Borel $T$-invariant probability measure $\mu_{\rk}$ on $X$. 
By \cite[Theorem 2.18]{Rud}, $\mu_{\rk}$ is regular. We now show that $\mu_{\rk}$ is an ergodic measure. Suppose that it is not ergodic. Then there is a $T$-invariant Borel subset 
$B$ of $X$ such that $\alpha := \mu_{\rk}(B)\in (0,1 )$. By regularity of the measure, and since the clopen subsets of $X$ form a basis for the topology, there are nonempty clopen subsets $\{U_i\}_{i \geq 1}$ in $X$ such that $\mu_{\rk}(B \triangle U_i) < \frac{1}{2^i}$ for all $i \geq 1$. We compute
\begin{align*}
\lim_{i \to \infty} \mu_{\rk}(U_i) & = \lim_{i \to \infty} \Big( \mu_{\rk}(U_i \backslash B) + \mu_{\rk}(U_i \cap B) \Big) \\
& = \lim_{i \to \infty} \mu_{\rk}(U_i \cap B) = \lim_{i \to \infty} \Big( \mu_{\rk}(U_i \cap B) + \mu_{\rk}(B \backslash U_i) \Big) = \mu_{\rk}(B).
\end{align*}
We then define $N_1(M) = \alpha^{-1} \lim_{i \to \infty}\rk(\chi_{U_i} M)$ and $N_2(M) = (1-\alpha)^{-1} \lim_{i \to \infty}\rk(\chi_{X \backslash U_i} M)$ for every matrix $M$ over $\calA$, and note that
$$N_1(1) = \alpha^{-1} \lim_{i \to \infty} \rk(\chi_{U_i}) = \alpha^{-1} \lim_{i \to \infty} \mu_{\rk}(U_i) = 1,$$
$$N_2(1) = (1-\alpha)^{-1} \lim_{i \to \infty}\rk(\chi_{X \backslash U_i}) = (1-\alpha)^{-1} \lim_{i \to \infty}\mu_{\rk}(X \backslash U_i) = 1,$$
From this and the approximate invariance of the sequence $\{ U_i \}_{i\ge 1}$, it is straightforward to check that each $N_i$ defines a Sylvester matrix rank function on $\calA$. To see that they are distinct Sylvester matrix rank functions, take $j \geq 1$ such that $\mu_{\rk}(B \triangle U_j) < \frac{1}{2} \text{min}\{\alpha, 1-\alpha\}$; then
$$N_1(\chi_{U_j}) = \alpha^{-1} \lim_{i \to \infty} \mu_{\rk}(U_i \cap U_j) = \alpha^{-1} \mu_{\rk}(U_j \cap B) > \frac{1}{2},$$
$$N_2(\chi_{U_j}) = (1-\alpha)^{-1} \lim_{i \to \infty} \mu_{\rk}((X \backslash U_i) \cap U_j) \leq (1-\alpha)^{-1} \mu_{\rk}(B \backslash U_j) < \frac{1}{2}.$$
Since $\rk = \alpha N_1 + (1-\alpha )N_2$, this contradicts the fact that $\rk$ is extremal in $\Ps(\calA)$.

Finally, the fullness of the measure follows from the faithfulness of $\rk$.
\end{proof}

\section{The space \texorpdfstring{$\Ps(\calA)$}{}}\label{section-space.PA}

In this section we obtain some results on the structure of the compact convex set $\Ps(\calA)$ of all the Sylvester matrix rank functions on $\calA$. Throughout this section, $T$ will denote a homeomorphism on a totally disconnected, metrizable compact space $X$, and $\calA = C_K(X) \rtimes_T \Z$.

Let $R$ be a unital ring. Following \cite{Jaik-survey}, we denote by $\Ps_{{\rm reg}}(R)$ the set of all Sylvester matrix rank functions $\rk$ on $R$ that are induced by some regular ring, that is, $\rk \in \Ps_{{\rm reg}}(R)$ if and only if there is a regular rank ring $(S, \rk_S)$ and a ring homomorphism $\varphi \colon R \to S$ such that $\rk(A) = \rk_S(\varphi (A))$ for every matrix $A$ over $R$.

We first investigate the relation between Sylvester matrix rank functions on $C_K(X)$ and Borel probability measures on $X$.

\begin{lemma}\label{lemma-ranks.on.CK(X)}
Let $T$ be a homeomorphism on a totally disconnected, metrizable compact space $X$. There is a natural identification $\Ps(C_K(X)) = M(X)$, where $M(X)$ denotes the compact convex set of probability measures on $X$. Under this identification, the set $M^{\Z}(X)$ of $T$-invariant probability measures corresponds to the set $\Ps^{\Z}(C_K(X))$ of $T$-invariant Sylvester matrix rank functions on $C_K(X)$. 
\end{lemma}

\begin{proof}
Note that $R = C_K(X)$ is a commutative von Neumann regular ring for each field $K$. Hence the set $\Ps(R)$ coincides with the set of pseudo-rank functions on $R$. Now it is clear that a pseudo-rank function $\rk$ on $R$ induces a finitely additive probability measure on the algebra of clopen subsets of $X$ by the rule $\ol{\mu}_{\rk}(U) = \rk(\chi_U)$ 
for every clopen subset $U$ of $X$, which by the same argument as in the proof of Proposition \ref{proposition-unique.rank}, can be uniquely extended to a Borel probability measure $\mu_{\rk}$ on $X$.

Conversely, any Borel probability measure $\mu$ induces a pseudo-rank function $\rk_{\mu}$ on $R$, as follows. Each element $a \in R$ can be written in the form $a = \sum _{i=1}^n \lambda_i \chi_{U_i}$, 
where $\lambda_i \in K$ and $\{ U_i \}_{i=1}^n$ forms a partition of $X$, where each $U_i$ is a clopen subset of $X$.  We define
$$\rk_{\mu}(a) = \sum_{i : \lambda_i \neq 0} \mu(U_i).$$
Let us check that it is indeed a pseudo-rank function on $R$. Clearly $\rk_{\mu}(0) = 0$ and $\rk_{\mu}(1) = 1$. If $a = \sum _{i=1}^n \lambda_i \chi_{U_i}, b = \sum _{j=1}^m \eta_j \chi_{V_j}$ 
with $\{U_i\}_{i=1}^n, \{V_j\}_{j=1}^m$ partitions of $X$ consisting of clopen sets, and $\lambda_i, \eta_j \in K$, then $ab = \sum_{i,j} \lambda_i \eta_j \chi_{U_i \cap V_j}$, and so
$$\rk_{\mu}(ab) = \sum_{i : \lambda_i \neq 0} \sum_{j : \eta_j \neq 0} \mu(U_i \cap V_j) \leq \sum_{i : \lambda_i \neq 0} \sum_j \mu(U_i \cap V_j) = \sum_{i : \lambda_i \neq 0} \mu(U_i) = \rk_{\mu}(a).$$
Symmetrically we get $\rk_{\mu}(ab) \leq \rk_{\mu}(b)$. To conclude, take $a = \chi_U$ and $b = \chi_V$ two orthogonal idempotents of $R$, so $U, V$ are disjoint clopen subsets of $X$. Then
$$\rk_{\mu}(a + b) = \mu(U) + \mu(V) = \rk_{\mu}(a) + \rk_{\mu}(b)$$
as required.

In this way, we obtain a canonical identification between $\Ps(R)$ and $M(X)$, since it is easily checked that $\rk_{\mu_{\rk}} = \rk$ and $\mu_{\rk_{\mu}} = \mu$. 
Now if $\mu$ is $T$-invariant and $a = \sum_{i=1}^n \lambda_i \chi_{U_i}$ is an element of $C_K(X)$ with $\{U_i\}_{i=1}^n$ a partition of $X$ consisting of clopen sets, then $T(a) = \sum_{i=1}^n \lambda_i \chi_{T(U_i)}$ and
$$\rk_{\mu}(T(a)) = \sum_{i : \lambda_i \neq 0} \mu(T(U_i)) = \sum_{i : \lambda_i \neq 0} \mu(U_i) = \rk_{\mu}(a).$$
Hence $\rk_{\mu}$ is $T$-invariant. Conversely, if $\rk$ is a $T$-invariant Sylvester matrix rank function, then
$$\mu_{\rk}(T(U)) = \rk(\chi_{T(U)}) = \rk(T(\chi_U)) = \rk(\chi_U) = \mu(U) \quad \text{ for every clopen subset $U$ of } X.$$
Since the extension to a Borel probability measure is unique, we conclude that $\mu$ is also $T$-invariant.
\end{proof}

\begin{proposition}\label{proposition-extr.points.are.regular}
Continue with the above notation. For each $\mu \in \partial_e M^{\Z}(X)$ there exists $\rk \in \partial_e \Ps(\calA) \cap \Ps_{{\rm reg}}(\calA)$ such that $\rk(\chi_U)= \mu(U)$ for all clopen subset $U$ of $X$.   
\end{proposition}

\begin{proof}
Note that, by \cite[Theorem 8.1.8]{Phillips}, $\partial_e M^{\Z}(X)$ is the set of ergodic $T$-invariant Borel probability measures on $X$. Now if $\mu \in \partial_e M^{\Z}(X)$, then following the observation given at the beginning of Section \ref{section-Sylvester.rank.functions.A}, 
\begin{enumerate}[a)]
\item either there is a periodic point $x \in X$ of $T$, of period $l \geq 1$, such that $\mu(\{x\}) = \frac{1}{l}$, and the support of $\mu$ is the orbit $\calO(x)$ of $x$, or
\item $X$ is atomless and the action is essentially free, in the sense that the set of periodic points is a $\mu$-null set (see \cite[Remark 2.3]{KL}).
\end{enumerate}
In the former case, we follow the idea given in the proof of Proposition \ref{proposition-periodic.point}. We construct a map $\rho : \calA \ra M_l(K)$ by sending

\[ f \in C_K(X), \text{ } f \mapsto \left( \begin{array}{@{}c@{}}
											\begin{matrix}
												f(x) & & & \bigzero \\
												  & f(T(x)) & & \\
													& & \ddots & \\
													\hspace*{0.2cm} \bigzero & & & f(T^{l-1}(x))
											\end{matrix}											
											\end{array}\right), \qquad t \mapsto \left( \begin{array}{@{}c@{}}
																									\begin{matrix}
																										0 & & & 0 & 1 \\
																										1	& 0 & & & 0 \\
																											& \ddots & \ddots & & \\
																											& & \ddots & 0 & \\
																										\hspace*{0.2cm} \bigzero & & & 1 & 0
																									\end{matrix}											
																									\end{array}\right) = u. \]
It can be verified that $u \rho(f) u^{-1} = \rho(T(f))$ by direct computation. It follows from the universal property of the crossed product that there is a unique algebra 
homomorphism $\rho : \calA \ra M_l(K)$ extending the above assignments. Now $M_l(K)$ is a regular rank ring, with unique normalized rank function $\rk_l$, so it induces a Sylvester matrix 
rank function on $\calA$ by the rule $\rk(a) = \rk_l(\rho (a))$ for $a \in \calA$. It is not difficult to see that the restriction of $\rk$ on $C_K(X)$ gives the Sylvester matrix rank function $\rk_{\mu}$ constructed in Lemma \ref{lemma-ranks.on.CK(X)}, and also that $\rk \in \partial_e \Ps(\calA) \cap \Ps_{{\rm reg}}(\calA)$.

In the latter case, we may restrict to the closed subspace $X':= \text{supp}(\mu)$ of $X$, which is an infinite, totally disconnected, compact metric space. Since $\mu$ is $T$-invariant, $T$ restricts to a homeomorphism of $X'$ and the restriction of $\mu$ to $X'$ is a full ergodic $T$-invariant probability measure. It follows from Theorem \ref{theorem-completion.A.Ay} and Proposition \ref{proposition-unique.rank} that there is $\rk \in \partial_e \Ps(\calA') \cap \Ps_{{\rm reg}}(\calA')$, where $\calA':= C_K(X') \rtimes_T \Z$, such that $\rk$ induces $\rk_{\mu}$ on $C_K(X)$. Considering the canonical projection
$$P : C_K(X) \rtimes _T \Z  \to C_K(X')\rtimes _T \Z,$$
we see that $\rk_{\calA} = \rk \circ P \in \partial_e \Ps(\calA) \cap \Ps_{{\rm reg}}(\calA)$, as desired. It is straightforward to check that $\rk_{\calA}$ satisfies the desired compatibility property with the measure $\mu$.
\end{proof}

\begin{remark}\label{remark-non.uniqueness}
In the case where $\mu \in \partial_e M^{\Z} (X)$ is a measure concentrated in the orbit of a periodic point, we cannot expect uniqueness of the extremal Sylvester matrix rank function on $\calA$ extending $\rk_{\mu}$, 
essentially because of the appearance of isotropy.

Consider, for example, the case of a fixed point $X = \{x\}$, with associated measure $\mu$ satisfying $\mu(\{x\})= 1$, and $K$ being any field of characteristic different from $2$. We obtain an extremal Sylvester matrix rank function $\rk'$ by pulling back the unique Sylvester matrix rank function on $K$ via the homomorphism
$$\calA \cong K[t,t^{-1}] \ra K[t,t^{-1}]/(t-\alpha) \cong K,$$
the first isomorphism given by $f \mapsto f(x)$, $t \mapsto t$, and $\alpha \in K \backslash \{0,1\}$. This Sylvester matrix rank function induces the same measure $\mu$ as in Proposition \ref{proposition-extr.points.are.regular}, but the rank functions are clearly different, since $\rk'(t-1) = 1$ and $\rk(t-1) = 0$. This shows the nonuniqueness statement above.
\end{remark}

To continue, we need the following result from \cite{JL}.

\begin{proposition}\label{proposition-from.JL}
Let $A=K[t,t^{-1}]$. Then $\Ps(A) = \Ps_{{\rm reg}}(A)$.
\end{proposition}

\begin{theorem}\label{theorem-regular.Sylvester.space}
Let $T$ be a homeomorphism on a totally disconnected compact metric space $X$ and set $\calA = C_K(X)\rtimes _T \Z$. Then we have $\Ps(\calA)= \Ps_{{\rm reg}}(\calA)$.
\end{theorem}

\begin{proof}
Let $\Ps^{\Z} (C_K(X)) = M^{\Z} (X)$ be the space of $T$-invariant measures on $X$, which we identify with the set of $T$-invariant Sylvester matrix rank functions on $C_K(X)$ (Lemma \ref{lemma-ranks.on.CK(X)}).
By \cite[Proposition 5.9]{Jaik}, it suffices to show that all extremal Sylvester matrix rank functions on $\calA$ are regular. Let $\rk \in \partial_e \Ps(\calA)$, and let $\mu_{\rk}$ be the ergodic, full, $T$-invariant probability measure on $X$ given by Proposition \ref{proposition-rank.measure}.

Assume first that $\mu_{\rk}$ is a measure concentrated in the orbit of a periodic point $x$, of period $l$. In this case, $\rk$ induces an extremal Sylvester matrix rank function on $C_K(\calO (x)) \rtimes_T \Z$, which is $*$-isomorphic to $M_l(K[t^l, t^{-l}])$ via the map
\[ f \in C_K(\calO (x)), \text{ } f \mapsto \left( \begin{array}{@{}c@{}}
											\begin{matrix}
												f(x) & & & \bigzero \\
												  & f(T(x)) & & \\
													& & \ddots & \\
													\hspace*{0.2cm} \bigzero & & & f(T^{l-1}(x))
											\end{matrix}											
											\end{array}\right), \qquad t \mapsto \left( \begin{array}{@{}c@{}}
																									\begin{matrix}
																										0 & & & 0 & t^l \\
																										1	& 0 & & & 0 \\
																											& \ddots & \ddots & & \\
																											& & \ddots & 0 & \\
																										\hspace*{0.2cm} \bigzero & & & 1 & 0
																									\end{matrix}											
																									\end{array}\right) \]
and so, by Proposition \ref{proposition-from.JL}, $\rk$ is a regular Sylvester matrix rank function.

If the support of $\mu_{\rk}$ is infinite then, since $\mu_{\rk}$ is an ergodic $T$-invariant measure, the arguments in Proposition \ref{proposition-extr.points.are.regular} apply to give that $\rk \in \Ps_{{\rm reg}}(\calA)$.

Thus in any case we get that $\rk \in \Ps_{{\rm reg}}(\calA)$, and the proof is complete.
\end{proof}

\section*{Acknowledgments}

The authors would like to thank Hanfeng Li for his helpful comments.

\end{document}